\documentclass[reqno]{amsart}
\usepackage{amssymb,graphicx,url}
\newtheorem{theorem}{Theorem}[section]
\newtheorem{corollary}[theorem]{Corollary}
\newtheorem{lemma}[theorem]{Lemma}
\newtheorem{proposition}[theorem]{Proposition}
\newtheorem*{defn}{Definition}

\newcommand{\BMO}{\textit{BMO}}
\newcommand{\C}{{\mathbb C}}
\newcommand{\charfn}{{1}}
\newcommand{\Id}{\text{Id}}
\newcommand{\R}{{\mathbb R}}
\newcommand{\Rd}{{{\mathbb R}^d}}
\newcommand{\Z}{{\mathbb Z}}
\newcommand{\eps}{\varepsilon}

\newcommand{\supp}{\operatorname{supp}}
\begin{document}

\title[Wavelet frame bijectivity]{Wavelet frame bijectivity on Lebesgue and Hardy spaces}
\author{H.-Q. Bui and R. S. Laugesen}
\address{Department of Mathematics, University of Canterbury,
  Christchurch 8020, New Zealand}
\email{Huy-Qui.Bui@canterbury.ac.nz}
\address{Department of Mathematics, University of Illinois, Urbana,
IL 61801, U.S.A.}
\email{Laugesen@illinois.edu}


\date{\today}
\dedicatory{Dedicated to our friend and mentor Guido Weiss.}
%
\begin{abstract}
We prove a sufficient condition for frame-type wavelet series in $L^p$, the Hardy space $H^1$, and $\BMO$. For example, functions in these spaces are shown to have expansions in terms of the Mexican hat wavelet, thus giving a strong answer to an old question of Meyer.

Bijectivity of the wavelet frame operator acting on Hardy space is established with the help of new frequency-domain estimates on the Calder\'{o}n--Zygmund constants of the frame kernel.
\end{abstract}
\maketitle

\section{\bf Introduction}
\label{intro}

One of the first wavelets ever considered was the Mexican hat (see Figure~\ref{mexican}), which is
\[
\psi(x) = (1-x^2) e^{-x^2/2} , \qquad x \in \R .
\]
\begin{figure}[t]
  \begin{minipage}{.45\linewidth} \hspace{.2in}
  \includegraphics[scale=0.4]{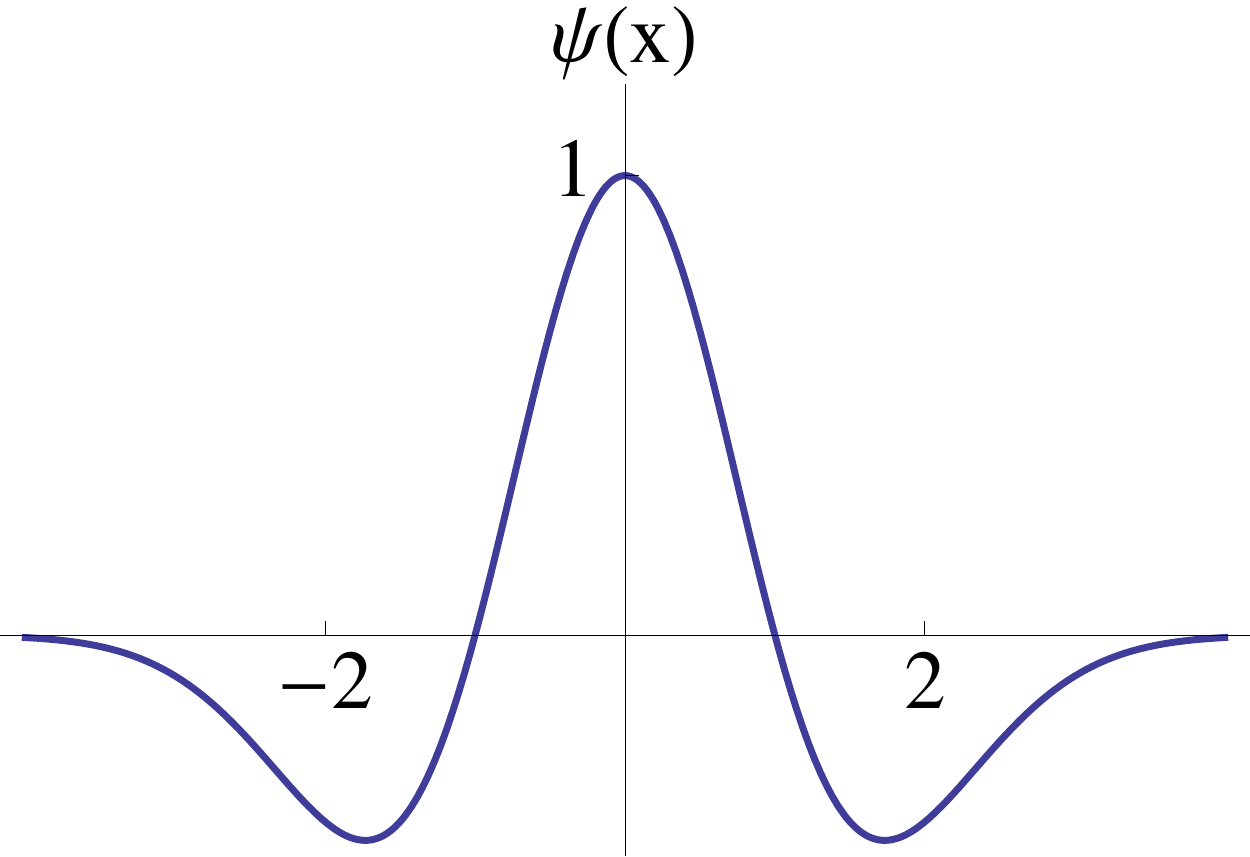}
  \end{minipage}
  \begin{minipage}{.45\linewidth} \hspace{.3in}
  \includegraphics[scale=0.4]{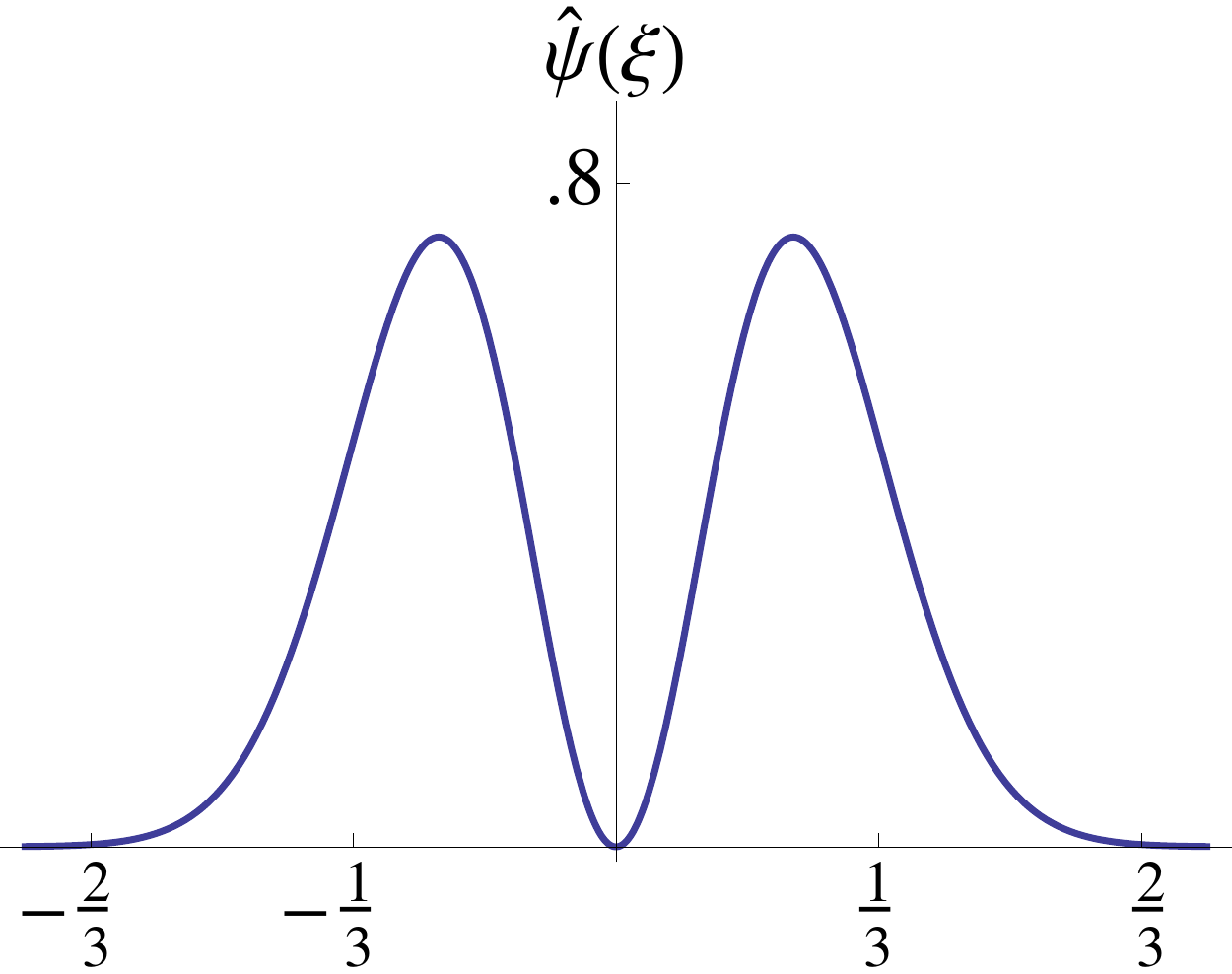}
  \end{minipage}
    \caption{\label{mexican}
    The Mexican hat function $\psi(x)=-\big( e^{-x^2/2} \big)^{\prime \prime}=(1-x^2) e^{-x^2/2}$,
    and its Fourier transform $\widehat{\psi}(\xi)=(2\pi \xi)^2 e^{-(2\pi)^2 \xi^2/2}$.}
\end{figure}
Morlet \cite{M83} generated the Mexican hat system
\[
 \psi_{j,k}(x) = 2^{j/2} \psi(2^j x - k) , \qquad j,k \in \Z ,
\]
and applied it to geophysical problems as if it were orthonormal (which it is not). Daubechies \cite[p.~75]{D92} explained why Morlet's approach succeeded: the system is an almost-tight \emph{frame}, with its frame operator
\[
f \mapsto  \sum_{j,k} \langle f , \psi_{j,k} \rangle \psi_{j,k}
\]
being a bijection on $L^2$, and indeed being rather close to the identity. In other words, analysis followed by synthesis with the Mexican hat system provides almost-perfect reconstruction in $L^2$. Hence the frame operator is surjective, and so every square integrable function can be decomposed into a norm convergent series in terms of the wavelet system $\psi_{j,k}$. That is, each $g \in L^2$ has an unconditionally convergent series expansion
\[
g = \sum_{j,k} c_{j,k} \psi_{j,k}
\]
for some coefficients $c_{j,k}$. Do such series expansions hold also when $g \in L^p$ for $p \neq 2$?

Meyer highlighted the extent of our ignorance on this issue when he remarked that \emph{``we do not know whether the functions $2^{j/2} \psi(2^j x - k), j, k \in \Z$, form a complete set in $L^p(\R)$ for $1<p<\infty$''} \cite[p.~137]{M92}. In other words, are finite linear combinations of the $\psi_{j,k}$ dense in $L^p$? Note that proving existence of series expansions in $L^p$ would be even stronger than completeness.

Two counterexamples reveal the subtlety of the problem. Tao \cite{TXX} showed for each $p \in (1,2)$ that there exists a smooth, compactly supported synthesizer $\psi$ that is ``good'' on $L^2$ (generating a wavelet frame there) but is ``bad'' on $L^p$ (having non-bijective frame operator). For $p \in(2,\infty)$, Tchamitchian and Lemari\'{e} \cite[pp.\ 130,136]{M92} found examples with even worse behavior. Their system $\{ \psi_{j,k} \}$ fails even to be complete in $L^p$.

Despite these obstacles, we recently answered Meyer's completeness question in the affirmative, by proving $L^p$-completeness criteria for general wavelet systems \cite{bl5,bl6}, and verifying those criteria for the specific Mexican hat example. Yet the stronger question of series expansions in $L^p$ remained open.

This paper answers the stronger question, by proving criteria for wavelet series expansions of arbitrary functions in $L^p, 1<p<\infty$, as well as in the Hardy space $H^1$ and $\BMO$ (see Theorem~\ref{invertible}). These bijectivity criteria on the frame operator are explicitly computable. We verify them numerically for the Mexican hat system, in Section~\ref{mexicanexample}.

To prove the main results (Theorem~\ref{invertible} and Corollary~\ref{cor-expansions}) we develop explicit norm estimates on general Calder\'{o}n--Zygmund operators (Proposition~\ref{CZatomicbound}) and computable bounds on the Calder\'{o}n--Zygmund constants for the wavelet frame kernel (Theorem~\ref{frameCZOfreq}).

An innovative feature of our work is that we estimate the wavelet frame kernel in the frequency domain. Prior approaches proceeded in the time domain: see \cite[Ch.\,5]{HW96}, \cite[p.\,165]{M92}, and related approaches \cite{G93,W99}. Estimates in the frequency domain are better for our purposes, for three reasons, First, we can estimate the Calder\'{o}n--Zygmund constants while keeping the analyzer and synthesizer together instead of splitting them apart (as discussed in Section~\ref{freqestimates}). Second, in the frequency domain we need not preserve cancelations, because $\widehat{\psi}$ is nonnegative for many examples (such as the Mexican hat in Figure~\ref{mexican}). Third, the construction of a ``good'' analyzer is easier in the frequency domain, where in order to obtain almost-perfect reconstruction of $L^2$ we need only satisfy a discrete Calder\'{o}n condition and ensure that the ``overlaps" are sufficiently small (see Section~\ref{mexicanexample}).

The techniques in this paper are largely unrelated to our earlier work \cite{bl5,bl6}, which relied on $L^2$-type methods in the frequency domain and did not use the Calder\'{o}n--Zygmund theory.

\subsection*{Wavelet definitions}
To state the main result, we need some definitions. Fix $A,B \in \R$ with
\[
|A|>1 \qquad \text{and} \qquad  B > 0.
\]
Call $A$ the \emph{dilation factor}, and $B$ the \emph{translation step}.
We rescale functions $\psi,\phi \in L^2$ by
\[
\psi_{j,k}(x) = |A|^{j/2} \psi(A^j x - Bk) , \qquad \phi_{j,k}(x) = |A|^{j/2} \phi(A^j x - Bk) , \qquad j,k \in \Z .
\]
The wavelet \emph{synthesis operator} associated with $\psi$ is the map
\[
c = \{ c_{j,k} \} \mapsto S(c) = \sum_{j,k \in \Z} c_{j,k} \psi_{j,k}
\]
where the coefficients $c_{j,k}$ are complex numbers. The \emph{analysis operator} associated with $\phi$ is the map
\[
f \mapsto T(f) = \{ \, B \langle f , \phi_{j,k} \rangle_{L^2} \, \}_{j, k \in \Z}
\]
where $\langle f , g \rangle_{L^2} = \int_\R f \, \overline{g} \, dx$. (The ``$B$'' in the definition of analysis appears only for later convenience. Some authors use $B \log |A|$ instead, to emphasize their Riemann sum interpretation with respect to the continuous parameter case; see \cite{D92}.) Sometimes we write $S_\psi$ and $T_\phi$, to emphasize the dependence of these operators on the underlying synthesizer and analyzer.

The \emph{wavelet frame operator} consists of analysis followed by synthesis:
\[
ST (f) = B \sum_{j,k \in \Z} \langle f , \phi_{j,k} \rangle_{L^2} \, \psi_{j,k} .
\]

\subsection*{Main results}
Write $X(\xi)=\xi$ for the identity function on $\R$, and $W^{k,p}=W^{k,p}(\R)$ for the usual Sobolev space. Also, write $f \lesssim g$ to mean $f \leq (\text{const.}) g$.

Two quantities $M_1(\psi,\phi)$ and $M_\infty(\psi,\phi)$ will be defined later, in Section~\ref{mainproofs}. Our main result says that if these quantities are less than $1$, then the wavelet frame operator is bijective on $H^1$ and $\BMO$ and hence on $L^p$ for $1<p<\infty$.
\begin{theorem}[Invertibility of frame operator] \label{invertible}
Assume $\widehat{\psi}, \widehat{\phi}, X\widehat{\psi}, X\widehat{\phi}  \in W^{3,1} \cap W^{3,2}$ with $\widehat{\psi}(0)=\widehat{\phi}(0)=0$, and that the derivatives decay near the origin and infinity according to
\begin{equation} \label{synthanal}
|\widehat{\phi}^{\, \prime}(\xi)| , \, |\widehat{\psi}^{\, \prime}(\xi)| \, \lesssim
\begin{cases}
|\xi|^\eps , & |\xi| \leq 1 , \\
|\xi|^{-\eps-5/2} , & |\xi| \geq 1 ,
\end{cases}
\end{equation}
for some $\eps>0$.

If $M_1(\psi,\phi) < 1$ and $M_\infty(\psi,\phi) < 1$, then $ST$ extends from $L^2$ to a bounded, linear and bijective operator on $H^1$, on $L^p$ for $p \in (1,\infty)$, and on $\BMO$.
\end{theorem}
The proof, in Section~\ref{mainproofs}, involves a quantitative perturbation estimate for wavelet frames: we  show $ST$ lies within distance $1$ of the identity operator on $H^1$ and hence is invertible there. Clearly this approach depends sensitively on the choice of norm for $H^1$. We employ the $L^2$-atomic norm, as defined in the following section. Of course, after the frame operator has been proved bijective, it remains bijective under any equivalent norm.

The constants $M_1$ and $M_\infty$ in Theorem~\ref{invertible} depend on two types of information: (i) Calderon--Zygmund bounds on the frame operator, which we control with frequency-domain estimates on the analyzer and synthesizer, and (ii) $L^2$-norm bounds on the frame operator, which we obtain in practice from Daubechies-type frame conditions. (The original Daubechies condition can be found in \cite{D92}. A later variant is due to Casazza and Christensen \cite{CC01}.) The precise dependence of $M_1$ and $M_\infty$ on these quantities is explained in the definitions in Section~\ref{mainproofs}, and then revisited in a specific situation in Section~\ref{mexicanexample}, when we treat the Mexican hat.

Next we give a strong answer to Meyer's completeness question about non-orthogonal wavelets, for the class of wavelets in our main theorem.
\begin{corollary}[Surjectivity of synthesis] \label{cor-expansions}
Assume all the hypotheses of Theorem~\ref{invertible} hold.

Then every function $f$ in $H^1$, or $\BMO$, or $L^p$ for $1<p<\infty$, can be expressed as a norm convergent wavelet series of the form $f=\sum_{j,k} c_{j,k} \psi_{j,k}$.
\end{corollary}
We apply the corollary in Section~\ref{mexicanexample} to the Mexican hat example. There we construct an analyzer $\phi$ which, when combined with the Mexican hat synthesizer, generates a frame operator lying less than $10^{-3}$ from the identity on $L^2$ and less than about $10^{-2}$ from the identity in $H^1$. Thus the frame operator is a bijection, and so the Mexican hat wavelet can reproduce each function in $H^1, L^p$ and $\BMO$.

Note that wavelet systems cannot reproduce $L^1$-functions with nonzero integral, because the wavelet $\psi$ and all its dilates have integral zero. Interestingly, reproduction in $L^p$ with $0 < p < 1$ can be achieved by using only the small scales (terms with $j>0$), as shown in \cite[\S4.4]{L08}.

Open problems that a motivated reader might pursue include higher dimensional analogues of Theorem~\ref{invertible}, and the case of $H^p$ spaces for $p<1$.

\subsection*{Relevant literature}
Motivating literature for our approach includes: the $L^2$ frame work of Chui and Shi \cite{CS93}, the phi-transform theory of band limited dual frames for Triebel--Lizorkin spaces by Frazier and Jawerth \cite{FJ85,FJ90,FJW91} with recent extension to irregular translations by Cabrelli, Molter and Romero \cite{CMR11}, the co-orbit theory of Feichtinger and Gr\"{o}chenig \cite{FG89a,FG89b,G91}, and the non-band limited but oversampled ``approximate duals'' of Gilbert, Han, Hogan, Lakey, Weiland and Weiss \cite{GHHLWW02}. Note that Gilbert \emph{et al.} treated translations in a lattice for the full scale of Triebel--Lizorkin spaces. Their work was extended to irregular translation grids (still oversampled) by Li and Sun \cite{LS12}, in the restricted case of $L^p,1<p<\infty$.

Underlying the approaches of Feichtinger and Gr\"{o}chenig and Gilbert \emph{et al.} (and Frazier and Jawerth \cite[\S4]{FJ90}, Bui and Paluszy\'{n}ski \cite{BP04}, and Li and Sun \cite{LS12}) is the fact that a sufficiently dense oversampling of translations and dilations must yield approximate reconstruction in any reasonable function space, for any reasonable analyzer and synthesizer. The reason is that by the Calder\'{o}n reproducing formula, perfect reconstruction holds in the limit of infinite oversampling (the case of continuous parameter translations and dilations). Unfortunately, those papers do not specify the degree of oversampling required to guarantee their conclusions. In other words, the dilation factor $A$ must lie sufficiently close to $1$ and the translation step $B$ sufficiently close to $0$, but \emph{how} close would constitute ``sufficiently close'' is simply not specified.

Theorem~\ref{invertible} avoids oversampling completely, by constructing a criterion that is usable for fixed values of $A$ and $B$. Further, the theorem relaxes the band limitation of Frazier and Jawerth, so that one may treat non-band limited examples such as the Mexican hat. The cost of that relaxation is merely that the frame operator yields approximate rather than perfect reconstruction: our frame operator is bijective rather than equalling the identity.

Let us now contrast our work with the established theory of \emph{orthonormal} wavelets. In that setting, the frame operator gives perfect reconstruction on $L^2$ and is bounded on other function spaces such as $H^1$ and $L^p$, assuming reasonable smoothness and decay of the wavelet. The frame operator therefore equals the identity on those spaces, simply by density of $L^2$, and thus one obtains wavelet expansions in whole families of function spaces \cite{HW96,M92}. In this current paper, we do not assume orthonormality or perfect reconstruction on $L^2$. We assume instead that the difference between the frame operator and the identity has small norm in $L^2$. The challenge in Theorem~\ref{invertible} is to show (under suitable hypotheses) that the norm of this difference remains small when considered in the Hardy space and $\BMO$, so that the frame operator is bijective on those spaces.

Our work is unrelated also to multiresolution analyses that generate pairs of perfectly reconstructing analyzers and synthesizers, such as \cite{CHS02,DHRS03,RS98}. Such methods do not apply to the Mexican hat wavelet, for example, because it satisfies no scaling or refinement equation. Theorem~\ref{invertible} takes a different approach: it finds sufficient conditions for a given analyzer and synthesizer to generate a frame-like structure.

\section{\bf Definitions}
\label{definitions}

We employ the Fourier transform with $2\pi$ in the exponent,
\[
\widehat{f}(\xi) = \int_\R f(x) e^{-2\pi i \xi x} \, dx , \qquad \xi \in \R ,
\]
so that Parseval's identity says $\langle f ,g \rangle_{L^2} = \langle
\widehat{f} , \widehat{g} \, \rangle_{L^2}$.

Next we specify the norms to be used on Hardy space and $\BMO$, and define the Calder\'{o}n--Zygmund operators that will act upon those spaces. As always, we write $L^p=L^p(\R)$.

\subsection*{Atomic Hardy space, and $\BMO$} We work with the Hardy space $H^1=H^1(\R)$ defined by $L^2$ atoms, as follows. Call $a \in L^2$ an \emph{atom} if $a$ is supported in some bounded interval $I$ with $\lVert a \rVert_{L^2} \leq |I|^{-1/2}$ and $\int_I a(x) \, dx = 0$. Notice $\lVert a \rVert_{L^1} \leq 1$. Define
\begin{align*}
H^1 & = \big\{ f \in L^1 : f=\sum_{j=1}^\infty \lambda_j a_j \ \text{for some atoms $a_j$ and} \\
&  \hspace*{5cm} \text{constants $\lambda_j \in \C$ satisfying} \ \sum_{j=1}^\infty |\lambda_j|<\infty  \big\} .
\end{align*}
(Note the series $f=\sum_j \lambda_j a_j$ converges in $L^1$.) Then $H^1$ is a Banach space under the atomic norm
\[
\lVert f \rVert_{H^1} = \inf \big\{ \sum_{j=1}^\infty |\lambda_j|<\infty :  f=\sum_{j=1}^\infty \lambda_j a_j \big\} ,
\]
and it imbeds continuously into $L^1$ with
\[
\lVert f \rVert_{L^1} \leq \lVert f \rVert_{H^1} .
\]

The dual of the Hardy space is $(H^1)^* = BMO$, the space of functions with bounded mean oscillation. We employ the dual space norm on $\BMO$:
\[
\lVert g \rVert_\BMO = \sup \{ |g(f)| : \text{$f \in H^1$ with $\lVert f \rVert_{H^1} = 1$} \} .
\]

\subsection*{Calder\'{o}n--Zygmund operators (CZOs)}
\begin{defn} \rm
Call $Z$ a \emph{generalized Calder\'{o}n--Zygmund operator} if Assumptions 1a, 1b, 1c, 2 and 3 below all hold, for some constants $C_1,C_2,C_3>0$.

Suppose $K(x,y)$ is a measurable, complex valued function on $\R \times \R$.

\emph{Assumption 1a.}
\begin{equation} \label{CZ1}
|K(x,y)| \leq C_1 \frac{1}{|x-y|} \qquad \text{whenever $x \neq y$.}
\end{equation}

\emph{Assumption 1b.}
\begin{equation} \label{CZ2}
|K(x_0,y)-K(x,y)| \leq C_2 \frac{|x-x_0|}{|x-y|^2} \qquad \text{whenever $|x-x_0| \leq \frac{1}{2} |x-y|$.}
\end{equation}

\emph{Assumption 1c.}
\begin{equation} \label{CZ3}
|K(x,y_0)-K(x,y)| \leq C_3 \frac{|y_0-y|}{|x-y|^2} \qquad \text{whenever $|y_0-y| \leq \frac{1}{2} |x-y|$.}
\end{equation}

\emph{Assumption 2.} $Z : L^2 \to L^2$ is a bounded linear operator such that if $f \in L^2$ has compact support then
\begin{equation} \label{assump2}
(Zf)(x) = \int_\R K(x,y) f(y) \, dy \qquad \text{whenever $x \notin \supp(f)$.}
\end{equation}
\emph{Assumption 3.} If $f \in L^2$ has compact support and integral zero, then $Zf \in L^1 \cap L^2$ has $\int_\R (Zf)(x) \, dx = 0$, and $Z^* f \in L^1 \cap L^2$ has $\int_\R (Z^* f)(x) \, dx = 0$.
\end{defn}
Note that the integral representation \eqref{assump2} need not hold when $x \in \supp(f)$, for example if $Z$ is the identity operator and $K \equiv 0$.

\noindent \emph{Remark.} This definition of generalized CZO follows the treatment by Hern\'{a}ndez and Weiss \cite{HW96}. A more general distributional definition was developed earlier by Coifman and Meyer \cite{MC97}; see also Grafakos \cite{G08b}.

\subsection*{A function used later}
Define
\[
D(\zeta) = 7 \sqrt{\frac{\zeta^2(\zeta^2+3)}{(\zeta^2-1)^3}} , \qquad \zeta > 1 ,
\]
so that $D(\zeta) \sim 7/\zeta$ as $\zeta \to \infty$.

\section{\bf CZOs on atomic Hardy space}
\label{boundedCZO}

Our wavelet results depend on boundedness of generalized Calder\'{o}n--Zygmund operators.
Boundedness of CZOs on $H^1,L^p$ and $\BMO$ is well known \cite{MC97,G08b,HW96}, but boundedness is not enough for our purposes. We need a certain operator norm less than $1$, in order to prove invertibility of the wavelet frame operator by using a Neumann series. Thus we must have \emph{explicitly computable} norm bounds on CZOs. The point of the next proposition is to prove such explicit bounds.

The constant used below is
\[
c(p) = \Big( 2^{2p/(p+1)} \frac{p(p+1)}{p-1} \Big)^{\! (2-p)(p+1)/2p} \sim \frac{4}{p-1} \qquad \text{as $p \searrow 1$.}
\]
Notice $c(2)=1$. Write $p^\prime$ for the conjugate exponent: $\frac{1}{p}+\frac{1}{p^\prime}=1$.

\begin{proposition} \label{CZatomicbound}
Assume $Z : L^2 \to L^2$ is a generalized Calderon--Zygmund operator with constants $C_1,C_2,C_3$. Fix $\zeta \geq 3$.

Then $Z$ has a bounded linear extension to $H^1$, to $L^p$ for each $p \in (1,\infty)$, and to $\BMO$, with norms:
\begin{align}
\lVert Z \rVert_{H^1 \to H^1} & \leq 2\sqrt{\zeta} \, \lVert Z \rVert_{L^2 \to L^2} + D(\zeta) C_3 , \label{eq-i} \\
\lVert Z \rVert_{L^p \to L^p} & \leq
c(p) \big( \sqrt{32 \zeta} \, \lVert Z \rVert_{L^2 \to L^2} + \frac{8\zeta}{\zeta^2-1} C_3 \big)^{(2/p) - 1} \lVert Z \rVert_{L^2 \to L^2}^{2 - (2/p)}  \qquad \text{when $1 < p \leq 2$,} \label{eq-ii} \\
\lVert Z \rVert_{L^p \to L^p} & \leq
c(p^\prime) \big( \sqrt{32 \zeta} \, \lVert Z \rVert_{L^2 \to L^2} + \frac{8\zeta}{\zeta^2-1} C_2 \big)^{1-(2/p)} \lVert Z \rVert_{L^2 \to L^2}^{2/p} \qquad \text{when $2 \leq p < \infty$,} \label{eq-iii}
\\
\lVert Z \rVert_{\BMO \to \BMO} & \leq 2\sqrt{\zeta} \, \lVert Z \rVert_{L^2 \to L^2} + D(\zeta) C_2 . \label{eq-iv}
\end{align}
Equality holds for $p=2$.
\end{proposition}
The purpose of including the free parameter $\zeta$ in the proposition is to improve the bounds by allowing  trade offs between the size of $\lVert Z \rVert_{L^2 \to L^2}$ and of $C_2$ and $C_3$. For instance, $\zeta=50$ is a nearly-optimal choice for the Mexican hat example in Section~\ref{mexicanexample}, whereas choosing $\zeta=3$ would give results an order of magnitude worse.

\medskip
In order to prove Proposition~\ref{CZatomicbound} at the end of this section, we will find explicit bounds saying (roughly) that generalized CZOs map atoms to molecules, and that molecules decompose into atoms, so that molecules belong to the Hardy space. The next two lemmas phrase these known results in a way that avoids the molecular norm, so as to prove the proposition more directly and obtain computable constants. Note that here, and in later proofs, we must not switch between different types of atoms, because doing so could introduce large additional constants into the norm estimates. We must consistently use $L^2$ atoms.

The first lemma says that $Z$ maps each atom to a function that decays like the square of the distance from the support of the atom.
\begin{lemma}[``CZOs map atoms to molecules''] \label{atomsmolecules}
Let $\zeta \geq 3$. If Assumptions 1a, 1c and 2 hold and $a(x)$ is an atom supported in a bounded interval $I$ centered at $y_0 \in \R$, then $Za \in L^1 \cap L^2$ with
\[
\big| (Za)(x) \big| \leq \frac{C_4 |I|}{(x-y_0)^2} , \qquad x \in (\zeta I)^c ,
\]
where the constant is
\[
C_4 =  \frac{1}{2\sqrt{3}} \sqrt{\frac{\zeta^4(\zeta^2+3)}{(\zeta^2-1)^3}} \, C_3 .
\]
\end{lemma}
Here $\zeta I$ denotes the interval having the same center as $I$ and $\zeta$ times the length.
\begin{proof}[Proof of Lemma~\ref{atomsmolecules}]
We adapt parts of the standard proof; see for example \cite[Theorem 5.6.8]{HW96}.

First, $Za \in L^2$ by Assumption 2, since the atom $a$ belongs to $L^2$.

Next, if $x$ lies outside the support $I$ of the atom $a$ then
\begin{align}
(Za)(x)
& = \int_I K(x,y) a(y) \, dy \qquad \text{by Assumption 2} \notag \\
& = \int_I [K(x,y) - K(x,y_0)]a(y) \, dy \label{eq-kerneldiff}
\end{align}
since $\int_I a(y) \, dy = 0$.

Now suppose $y \in I$ and $x \in (\zeta I)^c$. Then
\[
|y_0 - y| \leq \frac{1}{2} |I| \leq \frac{\zeta-1}{4} |I| \leq \frac{1}{2} |x-y| ,
\]
where the middle inequality uses that $\zeta \geq 3$. Assumption 1c then implies with \eqref{eq-kerneldiff} that
\begin{align}
\big| (Za)(x) \big|
& \leq C_3 \int_I \frac{|y_0 - y|}{(x-y)^2} |a(y)| \, dy \label{beforeCS} \\
& \leq C_3 \Big( \int_I \frac{(y_0 - y)^2}{(x-y)^4} \, dy \Big)^{\! 1/2} \lVert a \rVert_{L^2} \qquad \text{by Cauchy--Schwarz} \notag \\
& = C_3 \sqrt{\frac{2}{|I|}} \Big( \int_{-1}^1 \frac{z^2}{(w-z)^4} \, dz \Big)^{\! 1/2} \lVert a \rVert_{L^2} \notag
\end{align}
by a change of variable, where
\[
z = \frac{y-y_0}{|I|/2} \qquad \text{and} \qquad w = \frac{x-y_0}{|I|/2} .
\]
Evaluating the last integral shows that
\begin{align*}
\big| (Za)(x) \big|
& \leq C_3 \sqrt{\frac{2}{|I|}} \sqrt{\frac{2(w^2+3)}{3(w^2-1)^3}} \, \lVert a \rVert_{L^2} \\
& \leq C_3 \frac{2}{|I| \sqrt{3}} \sqrt{\frac{w^4(w^2+3)}{(w^2-1)^3}} \frac{1}{w^2}
\end{align*}
since the atom $a$ has $L^2$-norm at most $|I|^{-1/2}$. Note $|w| \geq \zeta$ (since $x \notin \zeta I$), and that the function $w \mapsto w^4(w^2+3)/(w^2-1)^3$ is decreasing for $w \geq 1$. Hence we may replace $w$ with $\zeta$, inside the square root in the last formula. The lemma now follows by substituting the definition of $w$ into the remaining factor of $1/w^2$.

Notice $Za$ is locally integrable because it belongs to $L^2$, and is globally integrable because it decays like $x^{-2}$ at infinity. Thus $Za \in L^1$.
\end{proof}

\begin{lemma}[``Molecules belong to Hardy space''] \label{moleculesatoms}
Suppose $M \in L^2$ and $I$ is a bounded interval centered at $y_0 \in \R$. Let $C_4>0$ and $\zeta \geq 3$.

If
\begin{equation} \label{decay}
\big| M(x) \big| \leq \frac{C_4 |I|}{(x-y_0)^2} , \qquad x \in (\zeta I)^c ,
\end{equation}
and $\int_\R M(x) \, dx = 0$, then $M \in H^1$ with norm
\[
\lVert M \rVert_{H^1} \leq 2 \sqrt{\zeta |I|} \, \lVert M \rVert_{L^2} + \frac{4}{3\zeta} \big( \sqrt{57} + \sqrt{320/3} \big) C_4 .
\]
\end{lemma}

The proof adapts the essential ideas from Coifman and Weiss's decomposition of a molecule into a sum of atoms \cite[Theorem C]{CW77}. They used dyadic scaling, $\gamma = 2$, whereas we obtain a better constant by optimizing over $\gamma>1$ in the proof below. Also, they chose the initial radius $R$ to depend on the $L^2$ norm of $M$, whereas our choice of $R$ will depend on $\zeta$ and $I$. But the central idea of decomposing $M$ into atoms supported on annuli comes directly from their work.

\begin{proof}[Proof of Lemma~\ref{moleculesatoms}]
We may suppose $M \not\equiv  0$, and that the interval is centered at $y_0=0$, by a translation. Let $\gamma > 1$. (We will later choose $\gamma=4$.) Define
\begin{align*}
R & = \frac{\zeta |I|}{2} , \\
I_j & = [-\gamma^j R,\gamma^j R] , \qquad j=0,1,2,3,\dots \\
\Omega_0 & = I_0 = [-R,R] = \zeta I , \\
\Omega_j & = I_j \setminus I_{j-1} = [-\gamma^j R,\gamma^j R] \setminus [-\gamma^{j-1} R,\gamma^{j-1}R]  , \qquad j=1,2,3,\dots \\
\charfn_j & = \charfn_{\Omega_j} , \\
M_j & = \Big( M - \frac{1}{|\Omega_j|} \int_{\Omega_j} M(y) \, dy \Big) \charfn_j .
\end{align*}
Clearly $M_j$ is supported in $\Omega_j \subset I_j$, and $M_j$ is square integrable and has mean value zero.

We break the proof into several steps.

\smallskip
Step 1. We will prove $M_j/\lambda_j$ is an atom, for numbers $\lambda_j$ that satisfy
\begin{align}
0 < \lambda_0 & \leq 2 \sqrt{\zeta |I|} \, \lVert M \rVert_{L^2} , \label{lambda0} \\
0 < \lambda_j & \leq \frac{4C_4}{\sqrt{3} \zeta} \sqrt{(\gamma-1)(\gamma^2+10\gamma+1)} \, \gamma^{-j} , \qquad j=1,2,3,\dots \label{lambdaj}
\end{align}
To start proving these claims, let
\[
\lambda_j = \lVert M_j \rVert_{L^2(\Omega_j)} \sqrt{|I_j|}
\]
so that $M_j/\lambda_j$ satisfies the definition of an atom supported in the interval $I_j$. (This choice of $\lambda_j$ is positive except when $M_j \equiv 0$, and in that exceptional case we may instead choose $\lambda_j$ to be any positive number satisfying \eqref{lambda0} or \eqref{lambdaj}.) Thus the task is to estimate the norm of $M_j$ and hence to establish \eqref{lambda0} and \eqref{lambdaj}. First notice that for $x \in \Omega_j$,
\begin{align}
|M_j(x)|
& \leq |M(x)| + \frac{1}{|\Omega_j|} \int_{\Omega_j} |M(y)| \, dy \label{basic1} \\
& \leq |M(x)| + \frac{1}{|\Omega_j|^{1/2}} \lVert M \rVert_{L^2(\Omega_j)} . \label{basic2}
\end{align}
Hence in particular, $\lVert M_0 \rVert_{L^2(\Omega_0)} \leq 2 \lVert M \rVert_{L^2(\Omega_0)}$ by \eqref{basic2} and the triangle inequality. Thus $\lambda_0 \leq 2\lVert M \rVert_{L^2} \sqrt{\zeta |I|} $, which is the desired estimate \eqref{lambda0}.

Now consider $j \geq 1$. Note that
\[
|I_j| = 2\gamma^j R \qquad \text{and} \qquad |\Omega_j| = 2(\gamma-1)\gamma^{j-1}R .
\]
If $x \in \Omega_j \subset (\zeta I)^c$ then
\begin{equation} \label{Mest}
|M(x)| \leq \frac{2RC_4}{\zeta x^2}
\end{equation}
by the hypothesis \eqref{decay} with $y_0=0$, and so by \eqref{basic1},
\begin{align*}
|M_j(x)|
& \leq \frac{2RC_4}{\zeta x^2} +  \frac{1}{(\gamma-1)\gamma^{j-1}R} \int_{\gamma^{j-1}R}^{\gamma^j R} \frac{2RC_4}{\zeta y^2} \, dy \\
& = \frac{2C_4}{\zeta} \Big( \frac{R}{x^2} + \frac{1}{\gamma^{2j-1}R} \Big) .
\end{align*}
Squaring and integrating shows that
\[
\lVert M_j \rVert_{L^2(\Omega_j)}^2 \leq \frac{8C_4^2}{3\zeta^2 \gamma^{3j}R} (\gamma-1)(\gamma^2+10\gamma+1) .
\]
Substituting this estimate into the definition of $\lambda_j$ implies the bound \eqref{lambdaj}.

\smallskip Step 2.
From Step 1 we conclude that $\sum_{j=0}^\infty M_j$ belongs to $H^1$, with norm
\[
\big\lVert \sum_{j=0}^\infty M_j \big\rVert_{H^1} \leq \sum_{j=0}^\infty \lambda_j \leq
2\sqrt{\zeta |I|} \, \lVert M \rVert_{L^2} + \frac{4C_4}{\zeta} \sqrt{\frac{\gamma^2+10\gamma+1}{3(\gamma-1)}} .
\]

\smallskip Step 3.
To construct the remaining parts of the decomposition of $M$, we let
\[
m_j = \int_{\Omega_j} M(y) \, dy \qquad \text{and} \qquad \widetilde{\charfn}_j = \frac{1}{|\Omega_j|} \charfn_j ,
\]
so that
\[
m_j \widetilde{\charfn}_j = \Big( \frac{1}{|\Omega_j|} \int_{\Omega_j} M(y) \, dy \Big) \charfn_j .
\]
We will show $\sum_{j=0}^\infty m_j \widetilde{\charfn}_j$ belongs to $H^1$.

Let $n_j = \sum_{k=j}^\infty m_k$. Then
\begin{align*}
\sum_{j=0}^\infty m_j \widetilde{\charfn}_j
& = \sum_{j=0}^\infty (n_j - n_{j+1}) \widetilde{\charfn}_j \\
& = \sum_{j=0}^\infty n_{j+1} (\widetilde{\charfn}_{j+1} - \widetilde{\charfn}_j)
\end{align*}
by summation by parts and using that $n_0=\int_\R M(y) \, dy = 0$ by hypothesis. Moreover,
\[
\sqrt{\frac{\gamma-1}{\gamma(\gamma+1)}} \, (\widetilde{\charfn}_{j+1} - \widetilde{\charfn}_j)
\]
is an atom supported in the interval $I_{j+1}$, because it has integral zero and because the disjointness of $\Omega_j$ and $\Omega_{j+1}$ implies that
\begin{align*}
\lVert \widetilde{\charfn}_{j+1} - \widetilde{\charfn}_j \rVert_{L^2(I_{j+1})}
& = \Big( \frac{1}{|\Omega_{j+1}|} + \frac{1}{|\Omega_j|} \Big)^{\! 1/2} \\
& = \sqrt{\frac{\gamma(\gamma+1)}{\gamma-1}} |I_{j+1}|^{-1/2} .
\end{align*}

Hence $\sum_{j=0}^\infty m_j \widetilde{\charfn}_j$ belongs to $H^1$, with atomic norm at most
\begin{equation} \label{nnormest}
 \sqrt{\frac{\gamma(\gamma+1)}{\gamma-1}} \sum_{j=0}^\infty |n_{j+1}| .
\end{equation}
Lastly, we observe that
\begin{align*}
|n_{j+1}|
& \leq \sum_{k=j+1}^\infty \int_{\Omega_k} |M(y)| \, dy \\
& \leq \frac{4RC_4}{\zeta} \int_{\gamma^j R}^\infty \frac{1}{y^2} \, dy \qquad \text{by estimate \eqref{Mest}} \\
& = \frac{4C_4}{\zeta} \gamma^{-j} .
\end{align*}
Hence by \eqref{nnormest} and a geometric sum, the atomic norm of $\sum_{j=0}^\infty m_j \widetilde{\charfn}_j$ is at most
\[
\frac{4C_4}{\zeta} \sqrt{\frac{\gamma^3(\gamma+1)}{(\gamma-1)^3}} .
\]

\smallskip
Step 4. Clearly we may decompose $M$ as
\[
M = \sum_{j=0}^\infty M_j + \sum_{j=0}^\infty m_j \widetilde{\charfn}_j ,
\]
and so the bounds from Step 2 and Step 3 show that
\[
\lVert M \rVert_{H^1} \leq 2\sqrt{\zeta |I|} \, \lVert M \rVert_{L^2} + \frac{4C_4}{\zeta}  \Big( \sqrt{\frac{\gamma^2+10\gamma+1}{3(\gamma-1)}} + \sqrt{\frac{\gamma^3(\gamma+1)}{(\gamma-1)^3}} \, \Big) .
\]
The right side is minimal when $\gamma \approx 4.6$. We choose $\gamma=4$, which lies close to the minimum point and gives a simple formula. The lemma follows immediately.
\end{proof}

Now we can deduce boundedness of $Z$ on atoms.
\begin{lemma}[CZOs are bounded on atoms] \label{atombound}
Suppose Assumptions 1a, 1c, 2, 3 hold, and that $a(x)$ is an atom. Let $\zeta \geq 3$. Then
$Za \in H^1$ with
\[
\lVert Za \rVert_{H^1} \leq 2 \sqrt{\zeta} \, \lVert Z \rVert_{L^2 \to L^2} +  D(\zeta) C_3 .
\]
\end{lemma}
\begin{proof}[Proof of Lemma~\ref{atombound}]
Since $a$ is an atom, we have $Za \in L^1 \cap L^2$ with $\int_\R (Za)(x) \, dx = 0$ by Assumption 3, and so
\[
|(Za)(x)| \leq \frac{C_4 |I|}{(x-y_0)^2}
\]
for all $x \in (\zeta I)^c$, by Lemma~\ref{atomsmolecules} (using Assumptions 1a, 1c and 2). Hence by Lemma~\ref{moleculesatoms}, we conclude $Za \in H^1$ with
\[
\lVert Za \rVert_{H^1}
\leq 2 \sqrt{\zeta |I|} \, \lVert Za \rVert_{L^2} + \frac{4}{3\zeta} \big( \sqrt{57} + \sqrt{320/3} \big) C_4 .
\]
Now we finish the proof by substituting for $C_4$ in terms of $C_3$ (using the formula in Lemma~\ref{atomsmolecules}) and recalling that $\lVert a \rVert_{L^2} \leq |I|^{-1/2}$.
\end{proof}

Next we develop a bound from $L^1$ to $\text{weak-}L^1$. The result is standard \cite[Theorem~8.2.1]{G08b}, and we include the lemma only to get a formula for the constant $C_5$.
\begin{lemma}[CZOs map $L^1$ to weak-$L^1$] \label{weakL1}
Let $\zeta \geq 3$. If Assumptions 1a, 1c, 2 hold and $f \in L^1 \cap L^2$, then
\begin{equation} \label{weakbound}
\big| \{ x \in \R : |(Zf)(x)| > \alpha \} \big| \leq \frac{C_5}{\alpha} \lVert f \rVert_{L^1} \qquad \text{for all $\alpha > 0$,}
\end{equation}
where $C_5 = \sqrt{32 \zeta} \, \lVert Z \rVert_{L^2 \to L^2} + \frac{8\zeta}{\zeta^2-1} C_3$. Hence $Z$ extends to a linear operator from $L^1$ to weak-$L^1$ that satisfies the same bound \eqref{weakbound}.
\end{lemma}
\begin{proof}[Proof of Lemma~\ref{weakL1}]
First we show that if $b \in L^2$ is supported in a bounded interval $I$ centered at $y_0 \in \R$ and $b$ has integral $0$, then
\begin{equation} \label{L1estimate}
\int_{(\zeta I)^c} \big| (Zb)(x) \big| \, dx \leq \frac{2\zeta}{\zeta^2-1} C_3 \lVert b \rVert_{L^1} .
\end{equation}
Indeed, by \eqref{beforeCS} we have
\[
\big| (Zb)(x) \big| \leq C_3 \int_I \frac{|y_0 - y|}{(x-y)^2} |b(y)| \, dy
\]
when $x \notin \zeta I$. Integrating gives
\[
\int_{(\zeta I)^c} \big| (Zb)(x) \big| \, dx
\leq C_3 \int_I \int_{(\zeta I)^c} \frac{|y_0 - y|}{(x-y)^2} \, dx \, |b(y)| \, dy .
\]
We evaluate the inner integral by changing variable like in the proof of Lemma~\ref{atomsmolecules}:
\begin{align*}
\int_{(\zeta I)^c} \frac{|y_0 - y|}{(x-y)^2} \, dx
& = \int_{\{ |w| \geq \zeta \}} \frac{|z|}{(w-z)^2} \, dw \\
& = \frac{2\zeta |z|}{\zeta^2-z^2} \leq \frac{2\zeta}{\zeta^2-1}
\end{align*}
since $|z| \leq 1$ when $y \in I$. Hence we have proved \eqref{L1estimate}.

To prove the lemma, we now apply the proof of \cite[Theorem~8.2.1]{G08b} to $f$, except changing the interval $Q$ in that proof to $I$, and changing the interval $Q^*$ to $\zeta I$, and using \eqref{L1estimate} as part of our estimate. One hence obtains the bound
\[
\big| \{ x \in \R : |(Zf)(x)| > \alpha \} \big| \leq \left( 8 \gamma \lVert Z \rVert_{L^2 \to L^2}^2 + \frac{\zeta}{\gamma} + \frac{8\zeta}{\zeta^2-1}C_3 \right) \frac{\lVert f \rVert_{L^1}}{\alpha}
\]
where $\gamma>0$ is a free parameter. (If we directly followed \cite{G08b} then we would obtain $2C_3$ instead of $C_3$, but we gain a factor of $2$ by noting that the ``bad'' function in the Calder\'{o}n--Zygmund decomposition has $L^1$-norm at most $2 \lVert f \rVert_{L^1}$.)

Minimizing over the choice of $\gamma$ yields the constant $C_5$ stated in the lemma.
\end{proof}

\subsection*{Proof of Proposition~\ref{CZatomicbound}}

Boundedness of $Z$ on $L^p$ for $p \in (1,2)$, as stated in estimate \eqref{eq-ii}, holds by combining the weak $(1,1)$ bound in Lemma~\ref{weakL1} and the strong $(2,2)$ bound in Assumption 2 with the explicit interpolation result in Proposition~\ref{interpr} (taking $r=1$ there).

Boundedness for $p \in (2,\infty)$ as in \eqref{eq-iii} follows by duality, because $Z^* : L^2 \to L^2$ satisfies Calder\'{o}n--Zygmund Assumptions 1, 2 and 3 for the kernel $K^*(x,y)=\overline{K(y,x)}$, except with the constants $C_2$ and $C_3$ interchanged.

Now we prove $Z$ is bounded on $H^1$. Let $f \in H^1$. Then $f$ belongs to $L^1$, and so $Zf$ belongs to weak-$L^1$ by Lemma~\ref{weakL1}. To obtain a formula for $Zf$, we consider an atomic decomposition $f=\sum_{j=1}^\infty \lambda_j a_j$ of $f$. Recall $Za_j \in H^1$ with its norm bounded by an absolute constant, by Lemma~\ref{atombound}. Hence the series $\sum_{j=1}^\infty \lambda_j Za_j$ converges absolutely in $H^1$. In fact, this series equals $Zf$ a.e.\ by an argument that uses the weak-$L^1$ result (see \cite[p.~95]{G08a}). Therefore the formula
\[
Zf = \sum_{j=1}^\infty \lambda_j Za_j
\]
holds, and holds for each choice of atomic decomposition. The $H^1$ norm estimate for $Z$ in \eqref{eq-i} now follows from the bound on atoms in Lemma~\ref{atombound}.

Boundedness of $Z$ on $\BMO$ as in \eqref{eq-iv} is immediate from duality and the boundedness of $Z^*$ on $H^1$.

\section{\bf Wavelet frame operators are CZOs}
\label{waveletsfreq}

The wavelet frame operator $ST$ is known to be a Calder\'{o}n--Zygmund operator, under suitable hypotheses on the synthesizer and analyzer. The existing proofs proceed in the spatial (or time) domain, using $L^1$ majorants: see  \cite[\S5.6]{HW96}, \cite[Chapter~7]{MC97} or the explicit version in \cite[Lemma 3.1]{LS12}. Unfortunately, in practice such time domain estimates are too big (by an order of magnitude) to prove invertibility for examples such as the Mexican hat. We need better estimates that can capture cancelations in the analyzer and synthesizer.

Theorem~\ref{frameCZOfreq} in this section develops improved estimates on those Calder\'{o}n--Zygmund constants of the wavelet frame operator, by working in the frequency domain and keeping $\widehat{\psi}$ and $\widehat{\phi}$ together there (rather than ``splitting apart'' $\psi$ and $\phi$ in the time domain, as one does in the standard approach). Then later we use the theorem to prove our main result, in Section~\ref{mainproofs}.

\subsection{Wavelet definitions}
Recall the \emph{frame operator} associated with synthesizer $\psi$ and analyzer $\phi$:
\[
ST (f) = \sum_{j,k \in \Z} B \langle f , \phi_{j,k} \rangle_{L^2} \, \psi_{j,k} .
\]
Formally, we may rewrite the frame operator as an integral operator
\[
(STf)(x) = \int_\R K(x,y) f(y) \, dy ,
\]
where the \emph{wavelet frame kernel} is defined formally by
\begin{equation} \label{framekerneldef}
K(x,y) = B \sum_{j,k \in \Z} \psi_{j,k}(x) \overline{\phi_{j,k}(y)} , \qquad x,y \in \R .
\end{equation}
Denote the portion of the kernel at level $j=0$ by
\[
K_0(x,y) = B \sum_{k \in \Z} \psi(x-Bk) \overline{\phi(y-Bk)} .
\]

\subsection{Kernel kernel estimates in the spatial domain}
We will prove absolute convergence of the series defining the kernel $K$, in the time domain, and of the series for its partial derivatives. Even though we ultimately want estimates in the frequency domain, the convergence of the kernel in the time domain ensures that the Calderon--Zygmund Assumptions 1 and 2 make sense.
\begin{lemma}[Convergence and differentiability of kernel at level $0$] \label{K0properties}
Assume $\psi$ is bounded and $|\phi(x)| \lesssim 1/(1+x^2)$. Then the series
\[
K_0(x,y) = B \sum_{k \in \Z} \psi(x-Bk) \overline{\phi(y-Bk)}
\]
is absolutely convergent, for each $x,y \in \R$. Further:

(i) If $\psi$ and $\phi$ are continuous, then so is $K_0$.

(ii) If $\psi$ is continuously differentiable, $\psi^\prime$ is bounded, and $\phi$ is continuous, then the partial derivative
\[
\frac{\partial K_0}{\partial x} = B \sum_{k \in \Z} \psi^\prime(x-Bk) \overline{\phi(y-Bk)}
\]
exists and is continuous.
\end{lemma}
The assumption $|\phi(x)| \lesssim 1/(1+x^2)$ has been chosen for its simplicity, and could certainly be weakened. It suffices to assume that $\phi$ has a bounded, symmetric decreasing, integrable majorant, such as $1/(1+|x|^{1+\delta})$. Similar comments apply to other lemmas below.

\emph{Note.} A formula for $\partial K_0/\partial y$ follows by interchanging the roles of $\psi$ and $\phi$.
\begin{proof}[Proof of Lemma~\ref{K0properties}]
For the absolute convergence, observe that
\[
|K_0(x,y)| \leq (\text{const.}) \lVert \psi \rVert_{L^\infty} \sum_{k \in \Z} 1 \big/ \big( 1 + (y-Bk)^2 \big) < \infty .
\]

(i) Now suppose in addition that $\psi$ and $\phi$ are continuous. Then $K_0(\tilde{x},\tilde{y}) \to K_0(x,y)$ as $(\tilde{x},\tilde{y}) \to (x,y)$ by an application of dominated convergence.

(ii) By definition of the partial derivative,
\begin{align*}
\frac{\partial K_0}{\partial x}
& = B \lim_{h \to 0} \sum_{k \in \Z} \frac{1}{h} \big[ \psi(x+h-Bk) - \psi(x-Bk) \big] \overline{\phi(y-Bk)} \\
& = B \sum_{k \in \Z} \psi^\prime(x-Bk) \overline{\phi(y-Bk)}
\end{align*}
using dominated convergence, since the difference quotient for $\psi$ is bounded by $\lVert \psi^\prime \rVert_{L^\infty}$. Continuity follows by part (i), since $\psi^\prime$ is continuous.
\end{proof}

Now we show that the series for the full kernel $K$ converges.
\begin{lemma}[Convergence and differentiability of kernel] \label{Kproperties}
Assume $|\psi(x)|$ and $|\phi(x)|$ are $\lesssim 1/(1+x^2)$. Then the series for $K(x,y)$ converges absolutely whenever $x \neq y$, with
\[
|K(x,y)| \leq B \sum_{j,k \in \Z} |A|^j |\psi(A^j x-Bk) \overline{\phi(A^j y-Bk)} | \leq (\text{const.}) \frac{1}{|x-y|} .
\]
Further:

(i) If $\psi$ and $\phi$ are continuous, then so is $K(x,y)$ wherever $x \neq y$. More precisely, the series $K(x,y) = \sum_{j \in \Z} |A|^j K_0(A^j x, A^j y)$ converges uniformly on the set $\{ (x,y) : |x-y| > \eps \}$, for each $\eps>0$, and hence $K$ is continuous there.

(ii) If $\psi$ is continuously differentiable, $\phi$ is continuous, and $|\psi^\prime(x)|$ and $|\phi(x)|$ are $\lesssim 1/(1+|x|^3)$, then the partial derivative $\partial K/\partial x$ exists by term-by-term differentiation wherever $x \neq y$, and $\partial K/\partial x$ is continuous. This term-by-term derivative converges absolutely, and satisfies
\[
\Big| \frac{\partial K}{\partial x}  \Big| \leq B \sum_{j,k \in \Z} |A|^{2j} \big| \psi^\prime(A^j x-Bk) \overline{\phi(A^j y-Bk)} \big| \leq (\text{const.}) \frac{1}{|x-y|^2} .
\]
\end{lemma}
Interchanging the roles of $\psi$ and $\phi$ gives an analogous formula for $\partial K/\partial y$.
\begin{proof}[Proof of Lemma~\ref{Kproperties}]
We adapt ideas from \cite[Lemma 5.3.12 and {\S}5.6]{HW96}; or see \cite{MC97}.

Write $\Gamma(x)=1/(1+x^2)$. For each $k$, the triangle inequality implies $|x-y| \leq |x-Bk|+|y-Bk|$, and so either $\frac{1}{2} |x-y| \leq |x-Bk|$ or else $\frac{1}{2} |x-y| \leq |y-Bk|$. Hence we may ``split apart'' $\psi$ and $\phi$, or rather the $\Gamma$-terms that bound them; that is,
\begin{align*}
|K_0(x,y)| & \leq C \sum_{k \in \Z} \Gamma(x-Bk) \Gamma(y-Bk) \\
& \leq C \Gamma \big( \frac{1}{2} |x-y| \big) \sum_{k \in \Z} \Gamma(y-Bk) + B \Gamma \big( \frac{1}{2} |x-y| \big) \sum_{k \in \Z} \Gamma(x-Bk) \\
& \leq C \Gamma \big( \frac{1}{2} |x-y| \big)
\end{align*}
where $C$ is some constant (which changes from line to line, and depends on $B$).

Summing the last estimate over $j$ gives that
\begin{align*}
|K(x,y)|
& \leq \sum_{j \in \Z} |A|^j \Gamma \big( A^j |x-y|/2 \big) \\
& \leq (\text{const.}) \frac{1}{|x-y|}
\end{align*}
by Lemma~\ref{Jsum} in the Appendix.

(i) Let $\eps>0$ and $|x-y|>\eps$. Since $\Gamma$ is symmetric decreasing, the proof above gives
\[
|K(x,y)| \leq \sum_{j \in \Z} |A|^j \Gamma \big( A^j \eps /2 \big) \leq (\text{const.}) \frac{1}{\eps} < \infty .
\]
Hence the series $K(x,y)=\sum_j |A|^j K_0(A^j x,A^j y)$ converges uniformly on the set where $|x-y| > \eps$. Since $K_0$ is continuous by Lemma~\ref{K0properties}(i), it follows that $K$ is continuous when $|x-y|>\eps$.

(ii) Write $\Gamma_1(x)=1/\big( 1 + |x|^3 \big)$ for the majorant of $\psi^\prime$ and $\phi$. By arguing as in the first part of this proof, one can show that the term-by-term derivative series $B \sum_{j,k \in \Z} |A|^j A^j \psi^\prime(A^j x-Bk) \overline{\phi(A^j y-Bk)}$ converges absolutely on $\{ x \neq y \}$, and is bounded by $(\text{const.}) / |x-y|^2$.

By Lemma~\ref{K0properties}(ii) we know that this term-by-term derivative equals
\[
\sum_{j \in \Z} \frac{\partial \ }{\partial x} \, |A|^j K_0(A^j x, A^j y)
\]
and that $K_0$ is continuously differentiable with respect to $x$. This last series converges uniformly when $|x-y|>\eps$, by arguments similar to part (i) above, and hence it converges to a continuous function. This continuous function is the derivative $\partial K/\partial x$.
\end{proof}

A function will possess majorants of the above type as soon as its Fourier transform possesses enough derivatives in $L^1$, as the next lemma summarizes.
\begin{lemma}[Sufficient conditions for majorants] \label{sobolevmajorant} \

If $\widehat{\psi} \in W^{2,1}$ then $\psi$ is continuous and $|\psi(x)| \lesssim 1/(1+x^2)$.

If $\widehat{\psi} \in W^{3,1}$ then $\psi$ is continuous and $|\psi(x)| \lesssim 1 \big/ \big( 1+|x|^3 \big)$.
\end{lemma}
\begin{proof}[Proof of Lemma~\ref{sobolevmajorant}]
From $\widehat{\psi} \in L^1$ we have that $\psi(x) = \int_\R \widehat{\psi}(\xi) e^{2\pi i \xi x} \, d\xi$ is continuous (possibly after redefining $\psi$ on a set of measure zero). Integrating by parts twice gives that
\[
\psi(x) = (2\pi ix)^{-2} \int_\R \widehat{\psi}^{\, \prime \prime}(\xi) e^{2\pi i \xi x} \, d\xi .
\]
Hence
\[
|\psi(x)| \leq \min \Big( \lVert \widehat{\psi} \rVert_{L^1} ,  |2\pi x|^{-2} \lVert \widehat{\psi}^{\, \prime \prime} \rVert_{L^1} \Big)  \lesssim \frac{1}{1+x^2} .
\]
If $\widehat{\psi} \in W^{3,1}$, then integrating by parts a third time yields a majorant that decays like $|x|^{-3}$ at infinity.
\end{proof}
The hypotheses in the last lemma can be weakened to assume just that $\widehat{\psi}$ belongs to an inhomogeneous Besov space, $B^2_{1,\infty}$ or $B^3_{1,\infty}$ respectively. Those weaker assumptions do not appear useful for our work, though, because the Besov space norm seems too complicated for explicit estimations to be practical.

\subsection{Wavelet frame operator on $L^2$--- CZ Assumption 2}
\begin{proposition} \label{frameCZOkernelrep}
Assume $\widehat{\psi}, \widehat{\phi} \in L^2$ decay near the origin and infinity according to
\begin{align}
|\widehat{\psi}(\xi)| & \lesssim
\begin{cases}
|\xi|^\eps , & |\xi| \leq 1 , \\
|\xi|^{-\eps -1/2} , & |\xi| \geq 1 ,
\end{cases} \label{synth} \\
|\widehat{\phi}(\xi)| & \lesssim
\begin{cases}
|\xi|^\eps , & |\xi| \leq 1 , \\
|\xi|^{-\eps -1/2} , & |\xi| \geq 1 ,
\end{cases} \label{anal}
\end{align}
for some $\eps>0$. Assume $\widehat{\psi}, \widehat{\phi} \in W^{2,1}$.

Then the frame operator  $ST$ satisfies Calder\'{o}n--Zygmund Assumption 2: it is linear and bounded on $L^2$, and if $f \in L^2$ has compact support and $x \notin \supp(f)$, then
\[
(STf)(x) = \int_\R K(x,y) f(y) \, dy .
\]
\end{proposition}
\begin{proof}[Proof of Proposition~\ref{frameCZOkernelrep}]
The frame operator $ST$ is linear and bounded on $L^2$, for example by \cite[Propositions~6 and 7]{bl5}.

We know $\psi$ and $\phi$ are $\lesssim 1/(1+x^2)$, by Lemma~\ref{sobolevmajorant} and the hypothesis that $\widehat{\psi}, \widehat{\phi} \in W^{2,1}$. Thus by Lemma~\ref{Kproperties}, the kernel satisfies $|K(x,y)| \lesssim 1/|x-y|$. If $f \in L^2$ has compact support and $x \not\in \supp(f)$, then the function $y \mapsto 1/|x-y|$ is bounded on the support of $f$, and so by dominated convergence,
\[
\int_\R K(x,y) f(y) \, dy
= B \sum_{j,k} \int_\R f(y) \overline{\phi_{j,k}(y)} \, dy \, \psi_{j,k}(x)
= (STf)(x) .
\]
\end{proof}

\subsection{Wavelet frame operator on weighted $\mathbf{L^2}$ --- CZ Assumption 3}
\begin{proposition} \label{frameCZOintzero}
Assume $\widehat{\psi}, \widehat{\phi} \in W^{1,2}$ with $\widehat{\psi}(0)=\widehat{\phi}(0)=0$, and that $\widehat{\psi}^{\, \prime}$ and $\widehat{\phi}^{\, \prime}$ decay near the origin and infinity according to \eqref{synthanal}.

Then the frame operator  $ST$ satisfies Calder\'{o}n--Zygmund Assumption 3:
if $f \in L^2$ has compact support and integral zero, then $STf  \in L^1 \cap L^2$ with $\int_\R (STf)(x) \, dx = 0$, and $(ST)^*f  \in L^1 \cap L^2$ with $\int_\R (ST)^*f \, dx = 0$.
\end{proposition}
\begin{proof}[Proof of Proposition~\ref{frameCZOintzero}] If $f \in L^2$ has compact support and integral zero, then its Fourier transform vanishes at the origin and belongs to the Sobolev space $W^{1,2}$. Hence by boundedness of analysis and synthesis on the space of such ``Littlewood--Paley functions'' (as in \cite[Propositions~5 and 6]{bl6}), the wavelet series defining $STf$ converges in the Littlewood--Paley space. Then the wavelet series converges also in $L^1$, by an easy imbedding \cite[formula (1)]{bl6}.

Integrating the series term-by-term gives that
\begin{align*}
\int_\R (STf)(x) \, dx
& = \sum_{j,k \in \Z} B \langle f , \phi_{j,k} \rangle_{L^2} \int_\R \psi_{j,k}(x) \, dx \\
& = 0
\end{align*}
because $\int_\R \psi(x) \, dx = \widehat{\psi}(0)= 0$. By interchanging the roles of $\phi$ and $\psi$ we see the same conclusion holds for the $L^2$-adjoint, since $(ST)^*=(S_\psi T_\phi)^*=S_\phi T_\psi$.
\end{proof}

\subsection{Wavelet frame kernel --- CZ Assumption 1} \label{freqestimates}
As explained at the beginning of the section, we need good, explicit estimates on the Calder\'{o}n--Zygmund constants of the wavelet frame kernel. Our estimates involve the following frequency-domain norm quantities:
\begin{align*}
\sigma_1(\psi,\phi) & = \sum_{l \in \Z} \lVert \widehat{\psi}(\cdot) \overline{\widehat{\phi}(\cdot+lB^{-1})} \rVert_{L^1} \\
\sigma_2(\psi,\phi) & = 2\pi \sum_{l \in \Z} \lVert X(\cdot) \widehat{\psi}(\cdot) \overline{\widehat{\phi}(\cdot+lB^{-1})} \rVert_{L^1} \\
\sigma_3(\psi,\phi) & = 2\pi \sum_{l \in \Z} \lVert X(\cdot) \widehat{\psi}(\cdot+lB^{-1}) \overline{\widehat{\phi}(\cdot)} \rVert_{L^1}
\end{align*}
and
\begin{align*}
\tau_1(\psi,\phi) & = \frac{1}{4\pi^2} \sum_{l \in \Z} \big\lVert \big[ \widehat{\psi}(\cdot) \overline{\widehat{\phi}(\cdot+lB^{-1})}\big]^{\prime \prime} \big\rVert_{L^1} \\
\tau_2(\psi,\phi) & = \frac{1}{4\pi^2} \sum_{l \in \Z} \big\lVert \big[ X(\cdot) \widehat{\psi}(\cdot) \overline{\widehat{\phi}(\cdot+lB^{-1})}\big]^{\prime \prime \prime} \big\rVert_{L^1} \\
\tau_3(\psi,\phi) & = \frac{1}{4\pi^2} \sum_{l \in \Z} \big\lVert \big[ X(\cdot) \widehat{\psi} (\cdot + lB^{-1})  \overline{\widehat{\phi}(\cdot)} \big]^{\prime \prime \prime} \big\rVert_{L^1}
\end{align*}
where $X(\xi)=\xi$ denotes the identity function.

Now we can estimate the Calder\'{o}n--Zygmund constants $C_1,C_2,C_3$.
\begin{theorem}[CZ Assumption 1] \label{frameCZOfreq}
Assume $\psi, \phi \in L^2$ with $\widehat{\psi}, \widehat{\phi} , X \widehat{\psi}, X \widehat{\phi} \in W^{3,1} \cap W^{3,2}$.

Then the wavelet frame kernel $K(x,y)$ in \eqref{framekerneldef} satisfies Calder\'{o}n--Zygmund Assumptions 1a, 1b, 1c with the constants
\begin{align*}
C_1(\psi,\phi) & = \frac{2|A|}{|A|-1} \sqrt{\sigma_1(\psi,\phi) \tau_1(\psi,\phi)} , \\
C_2(\psi,\phi) & = \frac{4|A| (2|A|+1)}{|A|^2-1} \sqrt[3]{\sigma_2(\psi,\phi) \tau_2(\psi,\phi)^2} , \\
C_3(\psi,\phi) & = \frac{4|A| (2|A|+1)}{|A|^2-1} \sqrt[3]{\sigma_3(\psi,\phi) \tau_3(\psi,\phi)^2} ,
\end{align*}
provided these constants are finite.
\end{theorem}
(\emph{Aside.} These constants blow up as $|A| \to 1$. This apparent singularity would disappear if we had multiplied the frame operator by $\log |A|$, as discussed in the Introduction. Anyhow, the issue is irrelevant  because we fix the dilation factor $A$ in this paper.)

The proof will build on off-diagonal decay lemmas for $K_0$ and its derivatives. The first lemma explains the central idea: we invoke Poisson summation to transfer to the frequency domain, and then integrate by parts.
\begin{lemma}[Decay of $K_0$] \label{kernel1}
Assume $\psi, \phi \in L^2$ with $\widehat{\psi}, \widehat{\phi} \in W^{2,1} \cap W^{2,2}$, and that $\sigma_1(\psi,\phi)<\infty$ and $\tau_1(\psi,\phi)<\infty$. Then for all $x,y \in \R$,
\[
|K_0(x,y)| \leq \min \Big\{ \sigma_1(\psi,\phi) ,  \frac{\tau_1(\psi,\phi)}{|x-y|^2} \Big\} .
\]
\end{lemma}
\begin{proof}[Proof of Lemma~\ref{kernel1}]
The hypotheses guarantee that $\psi$ and $\phi$ are continuous and are $\lesssim 1/(1+x^2)$, by Lemma~\ref{sobolevmajorant}. Hence the series for $K_0(x,y)$ converges absolutely and yields a continuous function, by Lemma~\ref{K0properties}. We have
\begin{align}
K_0(x,y)
& = B \sum_{k \in \Z} \psi(Bk+x) \overline{\phi}(Bk+y) \notag \\
& = \sum_{l \in \Z} \big( E_x \widehat{\psi} * E_y \widehat{\overline{\phi}} \, \big) (lB^{-1}) \label{eq-convolve}
\end{align}
by the Poisson summation formula (justified below), where $E_x(\xi)=e^{2\pi i \xi x}$. By writing out the convolution, we find
\begin{equation} \label{eq-K0convolve}
K_0(x,y) = \sum_{l \in \Z} e^{-2\pi i lB^{-1}y} \int_\R e^{2\pi i \xi (x-y)} \widehat{\psi}(\xi) \overline{\widehat{\phi}(\xi+lB^{-1})} \, d\xi .
\end{equation}
This last series for $K_0$ converges absolutely, with $|K_0(x,y)| \leq \sigma_1(\psi,\phi)$ by definition of $\sigma_1$.

Now we justify the use of Poisson summation. Let $F(t)=K_0(x+t,y+t)$, so that $F$ is continuous and $B$-periodic. The Fourier coefficients of $F$ have $\ell^1$-norm at most $\sigma_1(\psi,\phi)<\infty$, by the calculations above. Thus the Fourier series of $F$ is a continuous function, which therefore equals $F(t)$ for every $t$. Choosing $t=0$ gives the Poisson summation formula as used in \eqref{eq-convolve}, since $F(0)=K_0(x,y)$.

Returning now to the formula for $K_0(x,y)$ in \eqref{eq-K0convolve}, integrating by parts twice gives that
\[
K_0(x,y) = \frac{1}{\big( 2\pi i(x-y) \big)^2} \sum_{l \in \Z} e^{-2\pi i lB^{-1}y} \int_\R e^{2\pi i \xi (x-y)} \big[ \widehat{\psi}(\xi) \overline{\widehat{\phi}(\xi+lB^{-1})} \big]^{\prime \prime} \, d\xi ,
\]
which implies that $|K_0(x,y)| \leq \tau_1(\psi,\phi)/|x-y|^2$, and finishes the proof of the lemma.

\end{proof}
\begin{lemma}[Decay of $\partial K_0/\partial x$] \label{kernel2}
Assume $\psi, \phi \in L^2$ with $\widehat{\psi}, X \widehat{\psi}, \widehat{\phi} \in W^{2,1}$ and $X \widehat{\psi}, \widehat{\phi} \in W^{3,2}$, and that $\sigma_2(\psi,\phi)<\infty$ and $\tau_2(\psi,\phi)<\infty$. Then for all $x,y \in \R$,
\[
\left| \frac{\partial K_0}{\partial x} \right| \leq \min \Big\{ \sigma_2(\psi,\phi) , \frac{\tau_2(\psi,\phi)}{|x-y|^3} \Big\}.
\]
\end{lemma}
\begin{proof}[Proof of Lemma~\ref{kernel2}]
The hypotheses $\widehat{\psi}, X \widehat{\psi}, \widehat{\phi} \in W^{2,1}$ imply by Lemma~\ref{sobolevmajorant} that $\psi, \psi^\prime$ and $\phi$ are continuous and are $\lesssim 1/(1+x^2)$. Thus the series for $K_0$ and $\partial K_0/\partial x$ converge absolutely and yield continuous functions, by Lemma~\ref{K0properties}.

Since $\partial K_0/\partial x = B \sum_{k \in \Z} \psi^\prime(x-Bk) \overline{\phi(y-Bk)}$, the lemma follows by applying the proof of Lemma~\ref{kernel1} to $\psi^\prime$ instead of $\psi$, and integrating by parts three times instead of twice.
\end{proof}

\subsection*{Proof of Theorem~\ref{frameCZOfreq}} The hypotheses of the theorem ensure (by Lemma~\ref{sobolevmajorant}) that $\psi, \psi^\prime, \phi, \phi^\prime$ are all $\lesssim 1/(1+|x|^3)$. Recall that the kernel
\[
K(x,y) = \sum_{j \in \Z} |A|^j K_0(A^j x, A^j y) .
\]
has partial derivative
\[
\frac{\partial K}{\partial x} = \sum_{j \in \Z} |A|^j A^j \frac{\partial K_0}{\partial x}(A^j x, A^j y)
\]
by Lemma~\ref{Kproperties}(ii). The first Calder\'{o}n--Zygmund estimate \eqref{CZ1} follows by combining Lemmas~\ref{kernel1} and \ref{Jsum} (in the Appendix), with $z=x-y$ and with $C_1$ as defined in Theorem~\ref{frameCZOfreq}. Next,
\[
\left\vert \frac{\partial K}{\partial x} \right\vert \leq \frac{C_2}{4|x-y|^2}
\]
by combining Lemmas~\ref{kernel2} and \ref{Jsum}, with $C_2$ as defined in Theorem~\ref{frameCZOfreq}. This derivative bound implies the second Calder\'{o}n--Zygmund condition \eqref{CZ2}; for example, see \cite[pp.\,241--242]{HW96}. The third estimate \eqref{CZ3} follows by swapping the roles of $\psi$ and $\phi$. The theorem is now proved.

\section{\bf Proof of Theorem~\ref{invertible}: wavelet invertibility on Hardy space, $L^p$ and $\BMO$}
\label{mainproofs}

Recall the function $D(\zeta)$ defined for $\zeta>1$ in Section~\ref{definitions}, and the constants $C_1(\psi,\phi)$ and so on introduced in Section~\ref{freqestimates}. Define
\begin{align*}
N_1(\psi,\phi,\zeta)
& = 2\sqrt{\zeta} \lVert ST \rVert_{L^2 \to L^2} + D(\zeta) C_3(\psi,\phi) , \\
N_\infty(\psi,\phi,\zeta)
& = 2\sqrt{\zeta} \lVert ST \rVert_{L^2 \to L^2} + D(\zeta) C_2(\psi,\phi) .
\end{align*}
Discarding the terms with $D(\zeta)$ gives an inequality on the $L^2$ norm:
\begin{equation} \label{eq-N1help}
2 \lVert ST \rVert_{L^2 \to L^2} \leq \min \big\{ N_1(\psi,\phi,\zeta) , N_\infty(\psi,\phi,\zeta) \big\}.
\end{equation}
Obviously, if $N_1$ is finite for any value of $\zeta$ then it is finite for every value of $\zeta$, and similarly for $N_\infty$.
\begin{proposition}[Boundedness of the frame operator on the Hardy space and $\BMO$] \label{bounded}
Assume $\psi$ and $\phi$ satisfy the regularity and decay hypotheses in Theorem~\ref{invertible}, that is, $\widehat{\psi}, \widehat{\phi}, X\widehat{\psi}, X\widehat{\phi}  \in W^{3,1} \cap W^{3,2}$ and $\widehat{\psi}(0)=\widehat{\phi}(0)=0$, along with estimate \eqref{synthanal}. Let $\zeta \geq 3$ and assume $N_1(\psi,\phi,\zeta)<\infty$ and $N_\infty(\psi,\phi,\zeta)<\infty$.

Then the wavelet frame operator $ST$ on $L^2$ extends to a bounded linear operator on the Hardy space $H^1$ and on $\BMO$, with norm bounds:
\begin{align*}
\lVert ST \rVert_{H^1 \to H^1}
& \leq N_1(\psi,\phi,\zeta) , \\
\lVert ST \rVert_{\BMO \to \BMO}
& \leq N_\infty(\psi,\phi,\zeta) .
\end{align*}
\end{proposition}
\begin{proof}[Proof of Proposition~\ref{bounded}]
The frame operator $ST$ satisfies Calder\'{o}n--Zygmund Assumptions 1, 2 and 3, by Theorem~\ref{frameCZOfreq}, Proposition~\ref{frameCZOkernelrep} and Proposition~\ref{frameCZOintzero}, respectively. Now
Proposition~\ref{CZatomicbound} completes the proof.
\end{proof}

\subsection*{Defining the constants $M_1$ and $M_\infty$} Next we define the constants in the statement of Theorem~\ref{invertible}.

Assume $\psi$ and $\phi$ satisfy the regularity and decay hypotheses of Theorem~\ref{invertible} (as in Proposition~\ref{bounded}). Choose some other functions $\psi_*$ and $\phi_*$ that satisfy the same assumptions, and which in addition give perfect reconstruction on $L^2$, meaning
\begin{equation} \label{perfect}
S_{\psi_*} T_{\phi_*} = \Id : L^2 \to L^2 .
\end{equation}
(Such functions certainly exist, for example, smooth orthonormal wavelets with sufficiently many vanishing moments.)

Define the constant
\begin{equation} \label{eq-Mdef}
M_p(\psi,\phi) = \inf_{\zeta \geq 3} N_p(\psi - \psi_*,\phi,\zeta) + \inf_{\eta \geq 3} N_p(\psi_*,\phi - \phi_*,\eta) ,
\end{equation}
when $p=1,\infty$. (This notation suppresses the dependence of $M_p$ on our choices of $\psi_*$ and $\phi_*$, but no harm will come from that omission.)

Roughly speaking, if $M_p$ is ``small'' then the perturbations $\psi - \psi_*$ and $\phi - \phi_*$ are ``small''. Note that Theorem~\protect\ref{invertible} assumes $M_1$ and $M_\infty$ are both rather small, in fact, less than $1$ in magnitude. That means the perturbations are ``small'', so that $\psi$ and $\phi$ are ``close'' to giving perfect reconstruction.

\subsection*{Proof of the main result --- Theorem~\protect\ref{invertible}}
In view of the linearity of $S$ with respect to the synthesizer $\psi$, and linearity of $T$ with respect to $\phi$, the perfect reconstruction in \eqref{perfect} implies a decomposition
\begin{equation} \label{eq-decomp}
ST - \Id = S_{\psi - \psi_*} T_\phi + S_{\psi_*} T_{\phi - \phi_*}
\end{equation}
valid on $L^2$. Thus by Proposition~\ref{bounded}, $ST - \Id$ extends to a bounded operator on $H^1$ with norm bound
\[
\lVert ST - \Id \rVert_{H^1 \to H^1} \leq N_1(\psi - \psi_*,\phi,\zeta) + N_1(\psi_*,\phi - \phi_*,\eta)
\]
for each $\zeta,\eta \geq 3$. Taking the infimum of the right side with respect to $\zeta$ and $\eta$ gives $M_1(\psi,\phi)$, and so we have
\begin{align*}
\lVert ST - \Id \rVert_{H^1 \to H^1} \leq M_1(\psi,\phi) & < 1 , \\
\lVert ST - \Id \rVert_{L^2 \to L^2} \leq \frac{1}{2} M_1(\psi,\phi) & < \frac{1}{2} ,
\end{align*}
by hypothesis and \eqref{eq-N1help}.

We conclude the frame operator $ST$ is bijective on $H^1$ and on $L^2$, with inverse given by the norm convergent Neumann series
\[
(ST)^{-1} = (\Id - (\Id-ST))^{-1} = \sum_{k=0}^\infty (\Id-ST)^k .
\]
Notice this inverse operator is well defined on $H^1 + L^2$, because if $f \in H^1 \cap L^2$ then the inverse series $\sum_{k=0}^\infty (\Id-ST)^k f$ converges to the same limit function in $H^1$ as in $L^2$, by considering pointwise convergence of a subsequence of partial sums.

Since $ST$ and $(ST)^{-1}$ are bounded on $H^1$ and on $L^2$, they extend to bounded operators on $L^p,1<p<2$, by interpolation (either the complex method \cite[Corollary 1]{FS72}, or the real method \cite[Theorem III.6.1]{GR85}). These extended operators remain inverses of each other on $L^p$, by density of $L^p \cap L^2$. Thus $ST$ is bijective on $L^p$.

Argue similarly for $2<p<\infty$, using $N_\infty$ and $M_\infty$ and then interpolating between $L^2$ and $\BMO$ \cite[Corollary 2]{FS72}.

\smallskip
\noindent \emph{Note.} An alternative to decomposition \eqref{eq-decomp} is
\[
ST - \Id = S_\psi T_{\phi-\phi_*} + S_{\psi-\psi_*} T_{\phi_*} .
\]
If this formula is used then one obtains an alternative version of Theorem~\ref{invertible}, as one verifies by adapting the proof in the obvious manner.

\subsection*{Comments on the Neumann series in $L^p$}  The quantities $M_1$ and $M_\infty$ determine the rate of convergence of the Neumann series for $(ST)^{-1}$, in the proof above, for $H^1$ and $\BMO$. To get convergence rates on $L^p$ one may interpolate $\Id -ST$ as follows. (This next material is not needed elsewhere in the paper.) Define
\[
N_p(\psi,\phi,\zeta)
=
c(p) \big( \sqrt{32 \zeta} \, \lVert ST \rVert_{L^2 \to L^2} + \frac{8\zeta}{\zeta^2-1} C_3(\psi,\phi) \big)^{(2/p) - 1} \lVert ST \rVert_{L^2 \to L^2}^{2 - (2/p)}
\]
when $1 < p \leq 2$
and let $M_p(\psi,\phi)$ be as in \eqref{eq-Mdef}. Here $c(p)$ is the constant from Section~\ref{boundedCZO}. Then
\[
\lVert ST \rVert_{L^p \to L^p} \leq N_p(\psi,\phi,\zeta)
\]
by Proposition~\ref{CZatomicbound}, and so the decomposition \eqref{eq-decomp} implies $\lVert ST - \Id \rVert_{L^p \to L^p} \leq M_p(\psi,\phi)$. Hence if $M_p < 1$ then the Neumann series for $(ST)^{-1}$ converges with geometric rate $M_p(\psi,\phi)$.

This approach  breaks down for $p$ near $1$ because the interpolation constant $c(p)$ blows up there, and so one gets $M_p>1$. To obtain a convergence rate on the Neumann series for such $p$-values, one should instead interpolate \emph{powers} of $\Id -ST$, as explained in \cite{BL12b}.

\subsection*{Proof of Corollary~\protect\ref{cor-expansions}}
Formally, one merely notes that bijectivity of $ST$ on each function space in question (as given by Theorem~\ref{invertible}) implies surjectivity of $S$, which means that each function $f$ can be written as $\sum_{j,k} c_{j,k} \psi_{j,k}$ for some coefficients $c_{j,k}$.

To make this deduction rigorous, one needs $S$ and $T$ to be extended separately, whereas  Theorem~\ref{invertible} says only that the composition $ST$ extends from $L^2$ to be bounded on $H^1,\BMO$ and $L^p$. Fortunately, the extension of analysis and synthesis separately to those function spaces (and indeed to a whole scale of Triebel--Lizorkin spaces) has been proved by Frazier and Jawerth \cite[Theorems 3.5, 3.7, and p.{\,}81]{FJ90}. One can verify their hypotheses straightforwardly, given the assumptions in Theorem~\ref{invertible}. In particular they proved that $T : L^p \to \dot{f}^{0,2}_p$ and $S : \dot{f}^{0,2}_p \to L^p$ for a certain homogeneous Triebel--Lizorkin sequence space $\dot{f}^{0,2}_p$, and similarly for $H^1$ when $p=1$ and $\BMO$ when $p=\infty$.

Thus for all $f$ we have
\[
f = S \big( T (ST)^{-1} f \big) ,
\]
and so $f = \sum_{j,k} c_{j,k} \psi_{j,k}$ with unconditional convergence, where the coefficient sequence $\{ c_{j,k} \}=T \big( (ST)^{-1} f \big)$ belongs to the space $\dot{f}^{0,2}_p$.

\section{\bf Example: Mexican hat wavelet expansions in Hardy space, $L^p$ and $\BMO$}
\label{mexicanexample}

The Mexican hat synthesizer $\psi(x)=(1-x^2) e^{-x^2/2}$ is shown in Figure~\ref{mexican} along with its Fourier transform
\[
\widehat{\psi}(\xi) = (2\pi \xi)^2 \exp(-2\pi^2 \xi^2) .
\]
We assume dyadic dilations and unit translations, so that
\[
A=2, \qquad B = 1 .
\]

We wish to apply our results to this example, in particular to get expansions of arbitrary $f \in L^p$ in terms of the Mexican hat system $\{ \psi_{j,k} \}$. In order to apply our results, though, we must construct a suitable analyzer $\phi$.

\smallskip

1. First we will cut off the Mexican hat in the frequency domain, to obtain a band-limited synthesizer $\psi_*$. Towards that end, define a ``ramp'' function
\[
\rho(\xi) =
\begin{cases}
0 & \text{when $\xi \leq 0$,} \\
35 \xi^4 - 84 \xi^5 + 70 \xi^6 - 20 \xi^7 & \text{when $0 \leq \xi \leq 1$,} \\
1 & \text{when $\xi \geq 1$,}
\end{cases}
\]
so that $\rho \in C^3(\R)$ and $\rho^\prime(\xi)=140 \xi^3 (1-\xi)^3>0$ for $\xi \in (0,1)$. This formula for $\rho^\prime$ implies that $\rho(\xi)$ increases from $0$ to $1$ as $\xi$ increases from $0$ to $1$, and that
\begin{equation} \label{unity}
\rho(\xi) + \rho(1-\xi) = 1 \qquad \text{for all $\xi \in \R$.}
\end{equation}
Next define a cut-off function (see Figure~\ref{kappabeta}) by
\[
\kappa(\xi) =
\begin{cases}
\rho(6\xi-2) & \text{when $\xi \geq 0$,} \\
\kappa(-\xi) & \text{when $\xi \leq 0$,}
\end{cases}
\]
so that $\kappa \in C^3(\R)$ increases from $0$ to $1$ as $\xi$ increases from $1/3$ to $1/2$ and similarly as $\xi$ decreases from $-1/3$ to $-1/2$. Further define
\[
\widehat{\psi_*} = (1-\kappa) \widehat{\psi} .
\]
Obviously $\psi_*$ is band-limited, with $\widehat{\psi_*}$ supported in the interval $[-1/2,1/2]$. Notice $\widehat{\psi_*} = \widehat{\psi}$ on the interval $[-1/3,1/3]$, because $\kappa=0$ there.

\begin{figure}
  \begin{minipage}{.45\linewidth} \hspace{.2in}
  \includegraphics[scale=0.4]{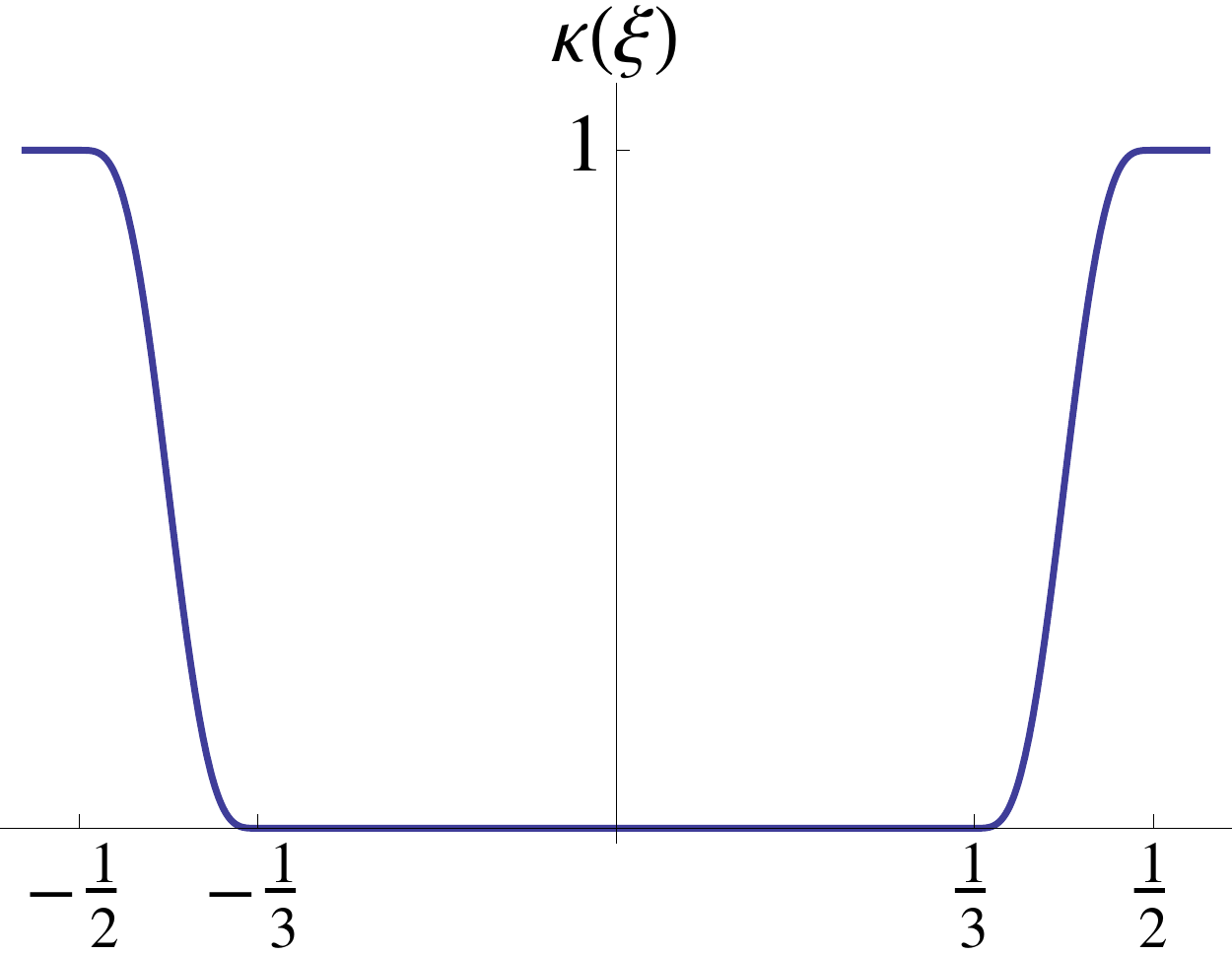}
  \end{minipage}
  \begin{minipage}{.45\linewidth} \hspace{.3in}
  \includegraphics[scale=0.4]{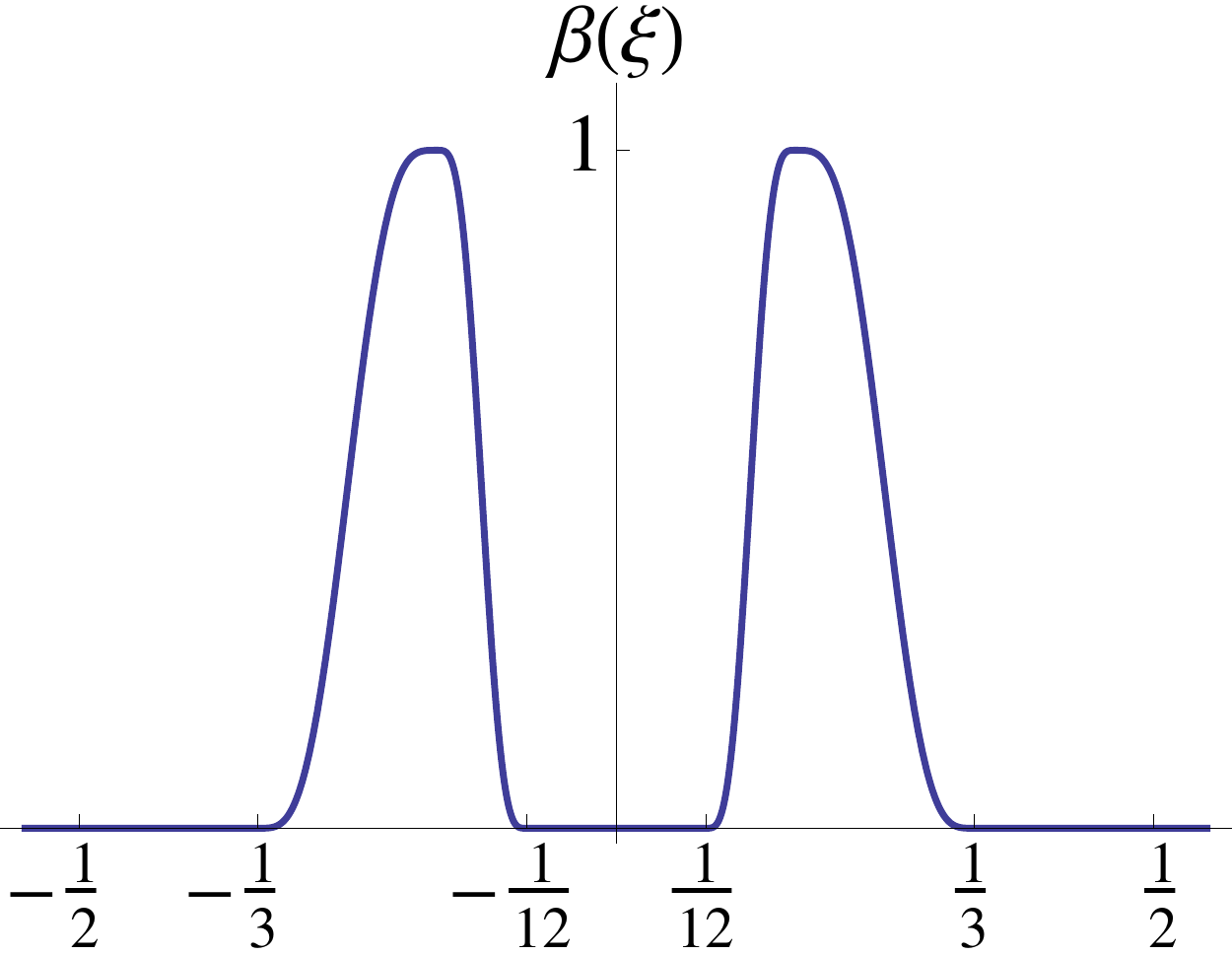}
  \end{minipage}
    \caption{\label{kappabeta}
    The cut-off function $\kappa$ and the double bump $\beta$, in the frequency domain.}
\end{figure}

\smallskip
2. Next we construct a $\phi$ such that analysis with $\phi$ followed by synthesis with $\psi_*$ gives perfect reconstruction on $L^2$. The method depends on a dyadic partition of unity, constructed as follows.

Start by defining a $C^3$-smooth ``double bump'' function supported in the interval $[-1/3,1/3]$:
\[
\beta(\xi) =
\begin{cases}
\rho(12\xi-1) & \text{when $0 \leq \xi \leq \frac{1}{6}$,} \\
\rho(2-6\xi) & \text{when $\xi \geq \frac{1}{6}$,} \\
\beta(-\xi), & \text{when $\xi \leq 0$,}
\end{cases}
\]
as plotted in Figure~\ref{kappabeta}. We claim that $\beta$ generates a dyadic partition of unity, with
\begin{equation} \label{pou}
\sum_{j \in \Z} \beta(2^j \xi) = 1 \qquad \text{for all $\xi \neq 0$.}
\end{equation}
Indeed, when $\frac{1}{12} \leq \xi < \frac{1}{6}$ we have $\beta(2^j \xi)=0$ for all $j \neq 0,1$ and the terms with $j=0,1$ give
\[
\beta(\xi) + \beta(2\xi) = \rho(12\xi-1) + \rho(2-12\xi) = 1
\]
by definition of $\beta$ and \eqref{unity}. (\emph{Note.} The double bump function $\beta$ is a particular case of the ``bell function'' construction used for Meyer--Lemari\'{e} wavelets \cite{HW96}.)

We define the analyzer $\phi$ by choosing its Fourier transform to be
\[
\widehat{\phi} = \beta/\widehat{\psi}
\]
as shown in Figure~\ref{muphi}, which is a kind of ``band-limited reciprocal'' of $\widehat{\psi}$. Hence $\widehat{\phi} \, \overline{\widehat{\psi_*}} = \beta$ on $\R$, because $\widehat{\psi_*}=\widehat{\psi}$ on the interval $[-1/3,1/3]$, which contains the support of $\beta$. Hence by the partition of unity property \eqref{pou}, we see $\psi_*$ and $\phi$ satisfy the ``discrete Calder\'{o}n condition''
\[
\sum_{j \in \Z} \widehat{\phi}(2^j \xi) \overline{\widehat{\psi_*}(2^j \xi)} = 1 ,  \qquad \xi \neq 0 .
\]
Since also $\widehat{\psi_*}$ and $\widehat{\phi}$ are supported in the interval $[-1/2,1/2]$ of length $1$, the discrete Calder\'{o}n condition implies that they give perfect reconstruction, meaning $S_{\psi_*} T_\phi = \Id$ on $L^2$.  (This fact, that band limitation together with discrete Calder\'{o}n implies perfect reconstruction, can be found in \cite[Chapter 6]{FJW91}, or else use \cite[formula (14)]{bl5}, for example.)

Choose $\phi_* = \phi$, so that $\psi_*$ and $\phi_*$ give perfect reconstruction on $L^2$. Then let
\[
\mu = \psi- \psi_* .
\]
Note that $\widehat{\mu} = \kappa \widehat{\psi}$. We plot $\widehat{\mu}(\xi)$ in Figure~\ref{muphi}.
\begin{figure}
  \begin{minipage}{.45\linewidth} \hspace{.2in}
  \includegraphics[scale=0.4]{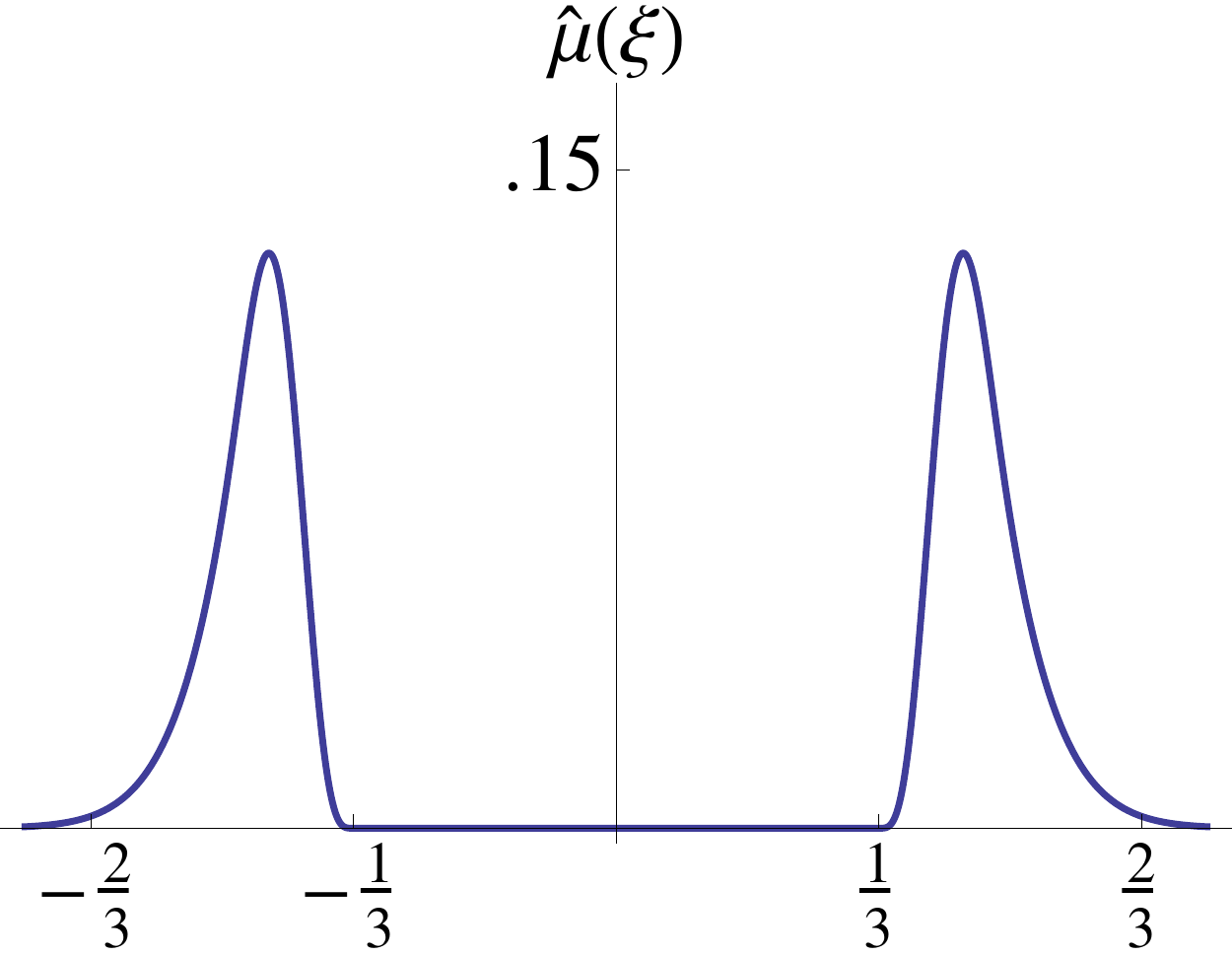}
  \end{minipage}
  \begin{minipage}{.45\linewidth} \hspace{.3in}
  \includegraphics[scale=0.4]{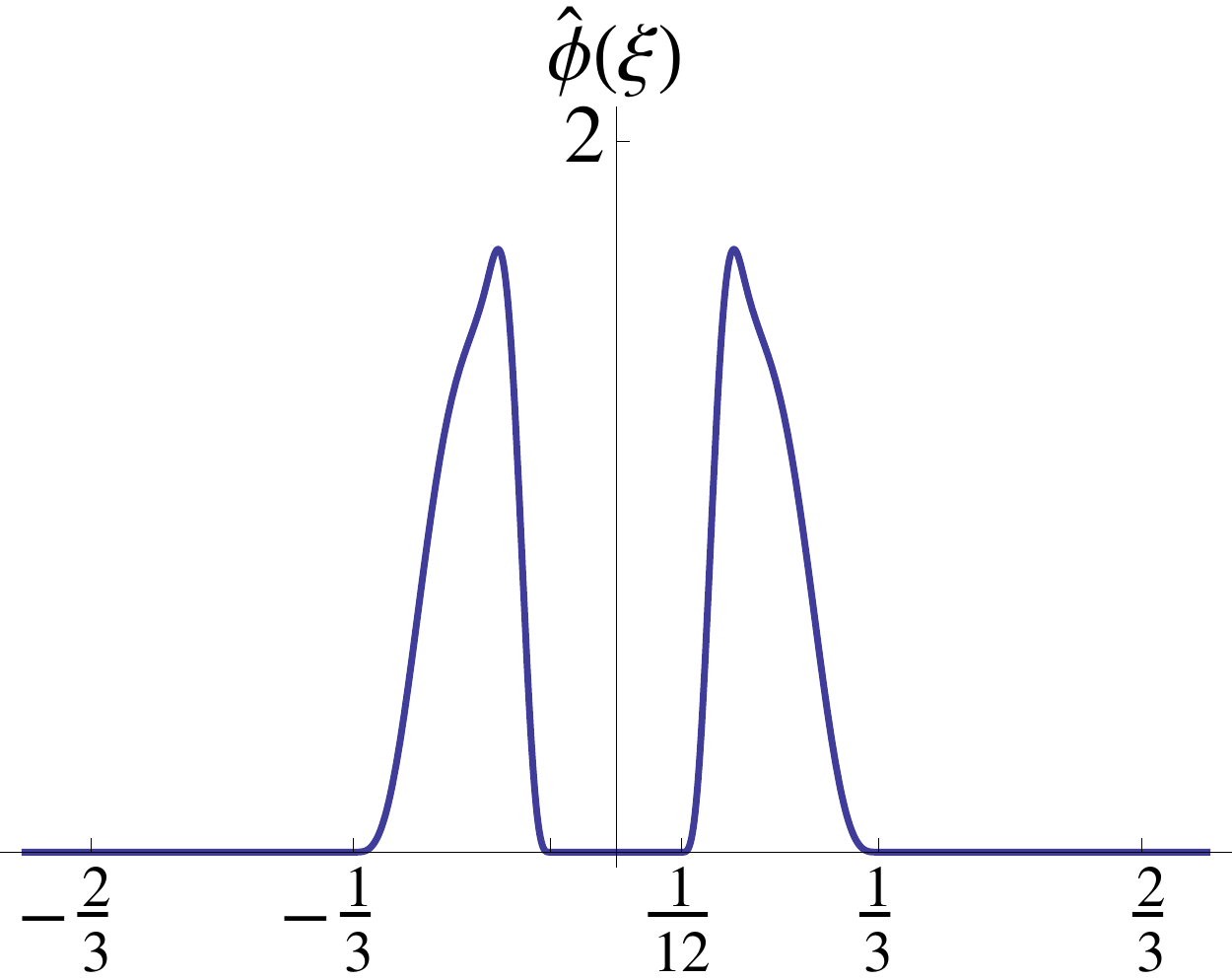}
  \end{minipage}
    \caption{\label{muphi}
    The Fourier transforms of the cut-off synthesizer, and of the analyzer.}
\end{figure}

\smallskip
3. Next, for the synthesizer $\mu$ and analyzer $\phi$ we estimate the quantities
\begin{align*}
\sigma_1(\mu,\phi) & = \sum_{l \neq 0} \lVert \beta(\cdot) \widehat{\psi}(\cdot+l) / \widehat{\psi}(\cdot) \rVert_{L^1[-1/3,1/3]} < 0.000045 \\
\sigma_2(\mu,\phi) & = 2\pi \sum_{l \neq 0} \lVert \beta(\cdot) X(\cdot+l) \widehat{\psi}(\cdot+l) / \widehat{\psi}(\cdot) \rVert_{L^1[-1/3,1/3]} < 0.00022 \\
\sigma_3(\mu,\phi) & = 2\pi \sum_{l \neq 0} \lVert \beta(\cdot) X(\cdot) \widehat{\psi}(\cdot + l) / \widehat{\psi}(\cdot) \rVert_{L^1[-1/3,1/3]} < 0.000067
\end{align*}
\begin{align*}
\tau_1(\mu,\phi) & = \frac{1}{4\pi^2} \sum_{l \neq 0} \big\lVert \big[ \beta(\cdot) \widehat{\psi}(\cdot+l) / \widehat{\psi} (\cdot) \big]^{\prime \prime} \big\rVert_{L^1[-1/3,1/3]} < 0.00086 \\
\tau_2(\mu,\phi) & = \frac{1}{4\pi^2} \sum_{l \neq 0} \big\lVert \big[ \beta(\cdot) X(\cdot+l) \widehat{\psi}(\cdot+l) / \widehat{\psi} (\cdot)\big]^{\prime \prime \prime} \big\rVert_{L^1[-1/3,1/3]} < 0.036 \\
\tau_3(\mu,\phi) & = \frac{1}{4\pi^2} \sum_{l \neq 0} \big\lVert \big[ \beta(\cdot) X(\cdot) \widehat{\psi} (\cdot + l) / \widehat{\psi}(\cdot) \big]^{\prime \prime \prime} \big\rVert_{L^1[-1/3,1/3]} < 0.014
\end{align*}
as we now justify. The first quantity $\sigma_1$ was defined in Section~\ref{freqestimates}, giving
\[
\sigma_1(\mu,\phi) = \sum_{l \in \Z} \lVert \widehat{\mu}(\cdot) \overline{\widehat{\phi}(\cdot+l)} \rVert_{L^1} = \sum_{l \in \Z} \lVert \widehat{\mu}(\cdot+l) \overline{\widehat{\phi}(\cdot)} \rVert_{L^1} .
\]
We may exclude $l=0$ from the sum because the supports of $\widehat{\mu}$ and $\widehat{\phi}$ do not overlap (see Figure~\ref{muphi}). Further, when $l \neq 0$ we may replace $\widehat{\mu}$ (which equals $\kappa \widehat{\psi}$) with $\widehat{\psi}$, because these functions agree outside the interval $[-1/2,1/2]$. Then we may substitute $\widehat{\phi}=\beta/\widehat{\psi}$. These steps lead to the formula given above for $\sigma_1(\mu,\phi)$, which we have estimated numerically using Mathematica to obtain the upper bound $0.000045$. We proceed similarly for each of the other quantities $\sigma_2,\sigma_3,\tau_1,\tau_2,\tau_3$.

From this numerical work we can estimate the Calder\'{o}n--Zygmund constants in Theorem~\ref{frameCZOfreq}, taking $A=2$ there:
\begin{align*}
C_1(\mu,\phi) & = 4 \sqrt{\sigma_1(\mu,\phi) \tau_1(\mu,\phi)} <  0.00079 , \\
C_2(\mu,\phi) & = \frac{40}{3} \sqrt[3]{\sigma_2(\mu,\phi) \tau_2(\mu,\phi)^2} <  0.088 , \\
C_3(\mu,\phi) & = \frac{40}{3} \sqrt[3]{\sigma_3(\mu,\phi) \tau_3(\mu,\phi)^2} < 0.032 .
\end{align*}

\smallskip 4. We have $S_\mu T_\phi = ST - \Id$ on $L^2$, since $\mu=\psi-\psi_*$ and $S_{\psi_*} T_\phi = \Id$ on $L^2$. Hence
\[
\lVert S_\mu T_\phi \rVert_{L^2 \to L^2}
= \lVert ST - \Id \rVert_{L^2 \to L^2} \leq \Delta
\]
where $\Delta$ denotes a Daubechies--type sufficient frame estimate; see \cite[Theorem 1]{bl5} with $p=2$. The quantity $\Delta$ can be evaluated numerically, giving $\lVert S_\mu T_\phi \rVert_{L^2 \to L^2} < 0.00026$.

Using this fact and our estimates from Step 3, we calculate from the definitions of $N_1$ and $N_\infty$ that
\begin{align*}
M_1(\psi,\phi) \leq N_1(\mu,\phi,90) & < 0.0075 , \\
M_\infty(\psi,\phi) \leq N_\infty(\mu,\phi,180) & < 0.011 ,
\end{align*}
where we chose the values of $\zeta=90$ and $\zeta=180$ to approximately minimize $N_1$ and $N_\infty$. (Recall also here that $\phi - \phi_* \equiv 0$, which eliminates the second term in the definitions of  $M_1$ and $M_\infty$.) Thus $ST$ provides almost-perfect reconstruction on $H^1$ and $\BMO$, with error of at most $1.1$\%.

Hence the operator $ST$ is bijective on $H^1$ and $\BMO$, and on $L^p$ for each $p \in (1,\infty)$, by Theorem~\ref{invertible}. We conclude from Corollary~\ref{cor-expansions} that every function in $L^p$ can be expressed as a norm convergent series in terms of the Mexican hat wavelet system $\{ \psi_{j,k} \}$. Thus our work provides a strong positive answer to Meyer's question in the Introduction.

\subsection*{Final comment on $L^p$}
Let us state some direct estimates for the frame operator on $L^p$, in the Mexican hat example, even though we do not need them for our work:
\[
M_p(\psi,\phi) \leq N_p(\mu,\phi,50) < 1 \qquad \text{when $1.04 \leq p \leq 2$.}
\]
When $1<p<1.04$ we cannot argue directly that the Neumann series for $(ST)^{-1}$ converges, because we cannot show $M_p<1$ in that range. That explains why in the example above we use Theorem~\ref{invertible}, whose proof interpolates the invertibility property between $H^1$ and $L^2$, rather than trying to prove directly that the Neumann series converges.

\section*{Acknowledgments} Laugesen thanks the Department of Mathematics and Statistics at the University of Canterbury, New Zealand, for hosting him during much of this research. This work was partially supported by a grant from the Simons Foundation (\#204296 to Richard Laugesen).

\appendix

\section{\bf Geometric sums}

In order to extend our decay estimates from $K_0$ to the full kernel $K$, earlier in the paper, we summed over dilation scales $j \in \Z$. Here we collect elementary estimates used in that task.
\begin{lemma} \label{Jsum}
Assume $g : \R \to \R$ and let $\sigma, \tau > 0$.

(i) If $|g(z)| \leq \min \{ \sigma, \tau/|z|^2 \}$ for all $z \neq 0$, then
\[
\sum_{j \in \Z} |A^j g(A^j z)| \leq \frac{2|A|}{|A|-1} \frac{\sqrt{\sigma \tau}}{|z|} , \qquad z \neq 0 .
\]

(ii) If $|g(z)| \leq \min \{ \sigma, \tau/|z|^3 \}$ for all $z \neq 0$, then
\[
\sum_{j \in \Z} |A^{2j} g(A^j z)| \leq \frac{|A|(2|A|+1)}{|A|^2-1} \frac{\sqrt[3]{\sigma \tau^2}}{|z|^2} , \qquad z \neq 0 .
\]
\end{lemma}
\begin{proof}[Proof of Lemma~\ref{Jsum}] \

(i) Fix $z \neq 0$ and let $J$ be any integer. We split the sum into two pieces and estimate each piece using the assumption on $g$:
\begin{align*}
\sum_{j \in \Z} |A^j g(A^j z)|
& \leq \sum_{j=-\infty}^{J-1} |A^j| \sigma + \sum_{j=J}^\infty |A^j| \frac{\tau}{|A^j z|^2} \\
& = \frac{|A|}{|A|-1} \big( |A|^{J-1} \sigma + |A|^{-J} |z|^{-2} \tau \big)
\end{align*}
by the geometric series. Choose $J$ to be the smallest integer satisfying $|A|^J |z| > \sqrt{\tau/\sigma}$, so that $|A|^{J-1} |z| \leq \sqrt{\tau/\sigma}$. (To motivate this choice, notice that $\sigma \leq \tau/|z|^2$ if and only if $|z| \leq \sqrt{\tau/\sigma}$.) Then
\[
\sum_{j \in \Z} |A^j g(A^j z)| \leq \frac{|A|}{|A|-1} \big( \sqrt{\tau/\sigma} |z|^{-1} \sigma + \sqrt{\sigma/\tau} |z|^{-1} \tau \big) ,
\]
from which part (i) of the lemma follows.

(ii) Adapt part (i), except choose $J$ to be the smallest integer satisfying $|A|^J |z| > \sqrt[3]{\tau/\sigma}$.
\end{proof}

\section{\bf Weak--strong interpolation between $L^1$ and $L^2$}

The Marcinkiewicz interpolation theorem implies in particular that if a sublinear operator satisfies weak-$L^r$ and weak-$L^2$ bounds, then strong-$L^p$ bounds hold for $p$ between $r$ and $2$. By strengthening the second hypothesis to strong-$L^2$ and calling also on Riesz--Thorin interpolation (which requires the operator to be linear), we will obtain an explicit estimate on the $L^p$ norm such that equality holds when $p=2$.

The bound will involve the following constants, for $1 \leq r < p \leq 2$:
\begin{align*}
c(p,r) & = \Big( 2^{2pr/(p+r)} \frac{p(p+r)(2-r)}{(p+r-pr)(p-r)} \Big)^{\! (2-p)(p+r)/2p(p+r-pr)} , \\
c(p) = c(p,1) & = \Big( 2^{2p/(p+1)} \frac{p(p+1)}{p-1} \Big)^{\! (2-p)(p+1)/2p} .
\end{align*}
Notice $p=2$ gives $c(2)=1$, and indeed $c(2,r)=1$ for each $r$. At the other extreme of $p$-values, we see $c(p,r)$ blows up like $(p-r)^{-1/r}$ as $p \downarrow r$, and $c(p)$ blows up like $(p-1)^{-1}$ as $p \downarrow 1$. See Figure~\ref{cpfig}.
\begin{figure}
  \includegraphics[scale=0.5]{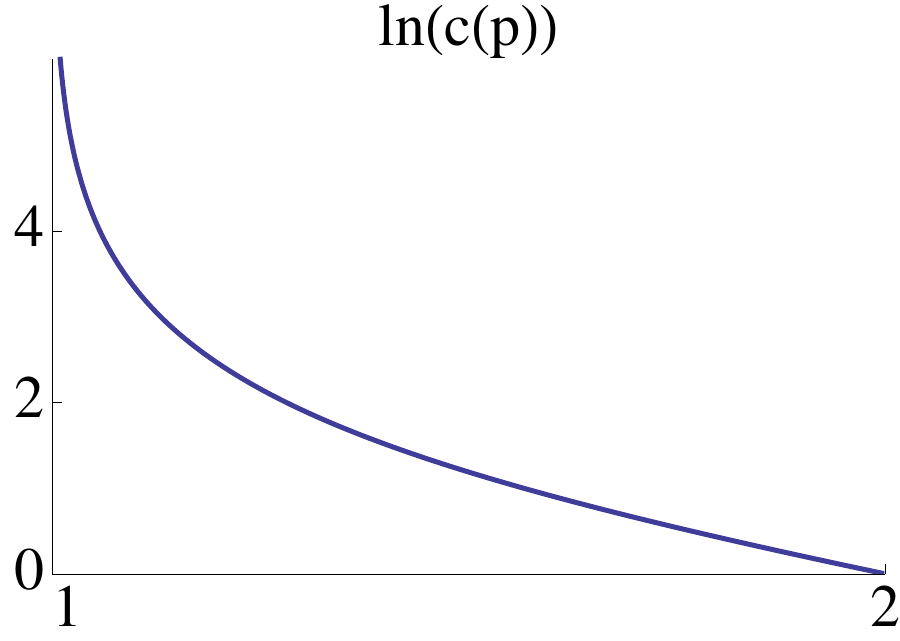}
    \caption{\label{cpfig}
    The logarithm of the $L^p$ constant in Proposition \ref{interpr}. Note $c(2)=1$.}
\end{figure}
\begin{proposition} \label{interpr}
Let $(X,\mu)$ and $(Y,\nu)$ be measure spaces, and $1 \leq r < 2$. Assume $Z$ is a linear operator from $L^r(X)+L^2(X)$ to the space of measurable complex-valued functions on $Y$.

If $Z$ is weak type $(r,r)$ and strong type $(2,2)$, then $Z$ is strong type $(p,p)$ for $r < p \leq 2$, with
\[
\lVert Z \rVert_{L^p(X) \to L^p(Y)} \leq c(p,r) \, \lVert Z \rVert_{L^r(X) \to \text{weak-}L^r(Y)}^{\left[(2/p) - 1\right]/\left[(2/r)-1 \right]} \, \lVert Z \rVert_{L^2(X) \to L^2(Y)}^{\left[ (2/r) - (2/p) \right]/\left[ (2/r)-1 \right]} .
\]
(Equality holds for $p=2$.) In particular, when $r=1$ the conclusion says for $1 < p \leq 2$ that
\[
\lVert Z \rVert_{L^p(X) \to L^p(Y)} \leq c(p) \, \lVert Z \rVert_{L^1(X) \to \text{weak-}L^1(Y)}^{(2/p) - 1} \, \lVert Z \rVert_{L^2(X) \to L^2(Y)}^{2 - (2/p)} .
\]
\end{proposition}
A similar result holds if we assume strong type $(q,q)$, but the proposition focuses on strong type $(2,2)$ because so many operators in harmonic analysis are bounded on $L^2$.

The case $r=1$ of the proposition improves by a factor of about $8$ (when $p$ is close to $2$) on the best bound we found in the literature, which is an exercise in the monograph by Grafakos \cite[Exercise 1.3.2]{G08a}. We follow the same approach as Grafakos, except in the proof below we employ the harmonic mean of $1$ and $p$ instead of the arithmetic mean; the harmonic mean yields simpler formulas.
\begin{proof}[Proof of Proposition~\ref{interpr}]
Suppose $r < q < p \leq 2$. Then Marcinkiewicz interpolation applied with $r<q<2$ (see \cite[Theorem 1.3.2]{G08a}) implies boundness of $Z$ on $L^q$, with
\begin{equation} \label{marcinbound}
\lVert Z \rVert_{L^q \to L^q} \leq 2 \Big( \frac{q}{q-r} + \frac{q}{2-q} \Big)^{\! \! 1/q} \, \lVert Z \rVert_{L^r \to \text{weak-}L^r}^{\left[(2/q) - 1\right]/\left[ (2/r)-1 \right]} \, \lVert Z \rVert_{L^2 \to L^2}^{\left[ (2/r) - (2/q) \right]/\left[ (2/r)-1 \right]} .
\end{equation}
Next, Riesz--Thorin interpolation applied with $q < p \leq 2$ (see \cite[Theorem 1.3.4]{G08a}) says that
\[
\lVert Z \rVert_{L^p \to L^p}
\leq \lVert Z \rVert_{L^q \to L^q}^{\left[ (2/p)-1 \right]/\left[ (2/q)-1 \right]} \lVert Z \rVert_{L^2 \to L^2}^{\left[ (2/q)-(2/p) \right]/\left[ (2/q)-1 \right]} .
\]
We substitute the Marcinkiewicz bound \eqref{marcinbound} into this last Riesz--Thorin bound, giving that
\begin{align}
\lVert Z \rVert_{L^p \to L^p}
& \leq \Big( \frac{2^q}{1-(r/q)} + \frac{2^q}{(2/q)-1} \Big)^{\! \left[ (2/p)-1\right] / (2-q)}  \label{interim} \\
& \qquad \qquad \qquad \times \lVert Z \rVert_{L^r \to \text{weak-}L^r}^{\left[(2/p) - 1\right]/\left[ (2/r)-1 \right]} \, \lVert Z \rVert_{L^2 \to L^2}^{\left[ (2/r) - (2/p) \right]/\left[ (2/r)-1 \right]} . \notag
\end{align}

We would like to choose $q \in (r,p]$ to minimize the right side of bound \eqref{interim}. That coefficient seems too complicated for analytical minimization to be feasible, but numerical work suggests that a good (\emph{i.e.}, within a factor of about $2$ of being minimal) choice for $q$ is the harmonic mean of $r$ and $p$:
\[
\frac{1}{q} = \frac{1}{2} \big( \frac{1}{r} + \frac{1}{p} \big) .
\]
Substituting this choice of $q$ into \eqref{interim} yields the constant $c(p,r)$ claimed in the proposition.

Notice equality holds when $p=2$, because equality holds in the Riesz--Thorin bound.
\end{proof}

\medskip

\end{document}